\documentclass[12pt]{amsart}
\usepackage{amsmath,amsthm,amssymb}

\usepackage[mathscr]{euscript}
\usepackage{xcolor} \usepackage{hyperref} 

\usepackage{enumitem}
\newtheorem{theorem}{Theorem}
\newtheorem{prop}[theorem]{Proposition}
\newtheorem{lemma}[theorem]{Lemma}

\newtheorem*{corstar}{Corollary}

\theoremstyle{definition}
\newtheorem{defn}{Definition}
\newtheorem{remark}{Remark}

\newcommand{\F}{{\EuScript F}}

\newcommand{\scV}{{\EuScript V}}

\newcommand{\PP}{\mathbb P}
\newcommand{\R}{\mathbb R}
\newcommand{\C}{\mathbb C}

\newcommand{\CP}{\C\PP}

\newcommand{\CH}{\C{\mathrm H}}
\newcommand{\RH}{\R{\mathrm H}}

\newcommand{\rhtwo} {\RH^2}

\newcommand{\cpn} {\CP^n}
\newcommand{\chn} {\CH^n}

\newcommand{\cpone} {\CP^1}
\newcommand{\cptwo} {\CP^2}
\newcommand{\chone} {\CH^1}
\newcommand{\chtwo} {\CH^2}

\newcommand{\ppt} {\mathrm p}
\newcommand{\qpt} {\mathrm q}
\newcommand{\xpt} {\mathrm x}

\def\X{{\mathscr X}} \def\({\left(}
\def\){\right)}
\def\<{\langle}
\def\>{\rangle}
\def\a {\alpha}
\def\b {\beta}
\def\l {\lambda}

\def\teps {\tilde {\epsilon}}

\def\sech{\text{sech}}
\def\csch{\text{csch}}
\def\bi{\bigskip}
\def\me{\medskip}

\def\g{\gamma}

\newcommand{\I}{{\mathcal I}}

\newcommand{\W}{{\EuScript W}}

\newcommand{\w}{\omega}

\newcommand{\setU}{\EuScript U}

\newcommand{\e}{\mathbf e}

\newcommand{\bv}{\mathbf v}

\newcommand{\bp}{\mathbf p}
\newcommand{\bq}{\mathbf q}

\newcommand{\bxi}{\boldsymbol \xi}

\newcommand{\bz}{\mathbf z}
\newcommand{\be}{\mathbf e}
\newcommand{\boldE}{\mathbf E}
\newcommand{\bR}{\mathbf R}
\newcommand{\bU}{\mathbf U}
\newcommand{\bdelta}{\boldsymbol \delta}
\newcommand{\bzeta}{\boldsymbol \zeta}

\newcommand{\ri}{\mathrm i}
\newcommand{\realpart}{\operatorname{Re}}
\newcommand{\imagpart}{\operatorname{Im}}

\newcommand{\JJ}{{\mathrm J}}

\newcommand{\restr}{\negthickspace \mid}

\newcommand{\mt}{\widetilde M}
\def\intprod{\mathbin{\raisebox{.4ex}{\hbox{\vrule height .5pt width 5pt depth 0pt \vrule height 3pt width .5pt depth 0pt}}}}

\def\&{\wedge}
\def\s{\sigma}
\def\a{\alpha}
\def\b{\beta}

\def\na{\nabla}
\def\nt{\widetilde \nabla}

\def\revisethingone{}
\def\unrevisethingone{\normalfont}
\def\revise#1{{\revisethingone {#1}}}\def\beginrevise{\revisethingone}
\def\endrevise{\unrevisethingone}
\def\revisethingtwo{}

\def\revisemore#1{{\revisethingtwo {#1}}}
\def\revisemorestill#1{{{#1}}}

\begin{document}
\title{Cohomogeneity-one ruled hypersurfaces in $\cptwo$ and $\chtwo$}
\author{Thomas A. Ivey}
\address{Dept. of Mathematics and Statistics, College of Charleston}
\email{iveyt@cofc.edu}
\author{Patrick J. Ryan}
\address{Department of Mathematics and Statistics, McMaster University}
\email{ryanpj@mcmaster.ca}
\date{\today}
\begin{abstract}
In this paper, we show how to construct a special class of ruled hypersurfaces in the nonflat complex space forms $\cpn$ and $\chn$.  This is done by taking an arbitrary smooth curve in  a totally geodesic (complex) one-dimensional submanifold and erecting an orthogonal ruling over each of its points.  Concentrating on the $n=2$ case, we also examine the special situation in which the base curve has constant geodesic curvature.  We show that, in this case, the construction yields precisely the real-analytic hypersurfaces of cohomogeneity one that satisfy certain transversality \revisemore{and nonvanishing conditions}.
\end{abstract}

\maketitle

\section{Introduction}
\label{intro}
\def\Ma{{\mathscr X}} \def\Mq{\widetilde{\mathscr X}} \def\Mh{{M}} \def\mpt{{\mathrm p}}  \def\qpt{{\mathrm q}}   \newcommand{\gam}{\mbox{\raisebox{.43ex}{$\boldsymbol\gamma$}}} \def\mysetminus{\setminus} 

Among real hypersurfaces in the complex space forms $\cpn$ and $\chn$, 
there are two interesting and contrasting classes, the Hopf hypersurfaces and the ruled hypersurfaces.
Each class may be defined by a simple condition on the shape operator (see \S\ref{background} for definitions).

The simplest Hopf hypersurfaces are the homogeneous ones and are characterized by constancy of their principal curvatures, \cite {kimura19xx}, \cite {takagi1973}, \cite{takagi1975}.
Although there has been a significant amount of study of ruled hypersurfaces since their introduction by Kimura \cite{kimura1987} in 1987, an identifiable class of ``simplest" ruled hypersurfaces has not emerged.  About 20 years ago, the second author constructed such a class of examples, but lacking an appropriate characterization, the construction never appeared in print.  
\revise{We now have \revisemore{a characterization theorem}, stated in detail in Section \ref{charT}.
While we state and prove these \revisemore{results} for $n=2$, the same conclusions hold in higher dimensions.}

Ruled hypersurfaces are foliated by codimension-one  ``rulings".  These are totally geodesic complex hypersurfaces of  $\cpn$ (resp. $\chn$) and are open subsets of complex projective spaces (resp. complex hyperbolic spaces) of complex dimension $n-1$.  Our construction proceeds by choosing a totally geodesic complex submanifold $M_0$ of dimension 1 (for example, a $\cpone$ in $\cpn$), taking a (real) curve in $M_0$, then ``erecting" a ruling over each point 
$\ppt$ on the curve.  This is done by choosing the unique totally geodesic complex hypersurface perpendicular to $M_0$ at $\ppt$.  

We know that $\cpone$ (resp. $\chone$) is isometric to a 2-sphere of constant positive curvature (resp. a real hyperbolic plane of constant negative curvature).  If our curve is chosen to be of constant geodesic curvature, our construction 
(see \S \ref{coho-1-Construction}) 
gives rise to the ``simplest" examples.  These hypersurfaces satisfy the following properties:
\begin{enumerate}
\item They are real-analytic ruled hypersurfaces;
\item They are of cohomogeneity one;
\item They satisfy \revisemorestill{certain} transversality and non-vanishing conditions.
\end{enumerate}
For clarification of this terminology, see \S \ref{background}.

\vspace{0.5cm}

This paper presents two major results:
\begin{itemize}
\item
The General Case:  Construction of a ruled hypersurface over an arbitrary plane curve $\gam$ and determination of its characteristic parameters $\a$ and $\b$ in terms of $\gam$ (sections 1--5);
\item
The Cohomogeneity-One Case:  
Description of the hypersurfaces resulting when $\gam$ has constant geodesic curvature and proof that these are precisely the ruled hypersurfaces satisfying the three properties mentioned above
(sections 6--7).
\end{itemize}

In this paper, 
all manifolds are assumed connected and all manifolds and maps are assumed
smooth (class $C^\infty$) unless otherwise stated.  Notation and definitions
not explicitly introduced may be found in \cite{KN} or \cite{nrsurvey}.  For general background on hypersurfaces in complex space forms (resp. moving frames and exterior differential systems), we refer the reader to \cite{crbook} (resp. \cite{CfB2}).

\subsubsection*{Acknowledgements}
\revise{The authors thank the referee for bringing to our attention the recent results of Dom\'inguez-V\'azquez and
P\'erez-Barral \cite{dominguez} on ruled hypersurfaces.}  In addition, the second author would like to acknowledge valuable conversations with McKenzie Wang. 

\revise{
\begin{remark}
It is worth comparing the results in \cite{dominguez} to ours. 
In \cite{dominguez} it is proved that if $M$ is a connected ruled real hypersurface in a complex space form of constant holomorphic sectional curvature $4c$, and the shape operator $A$ 
of $M$ has constant norm, then $M$ lies in $\CH^n$ and is an open subset of either a homogeneous Lohnherr hypersurface 
(see, e.g., p.524 in \cite {crbook})
or a hypersurface constructed (as described above) over a curve $\gam$ in $\CH^1$ of constant geodesic 
curvature $\kappa_g$ equal to the specific value of $\sqrt{-2c}$.  By contrast, in this article we assume that $M$ is ruled and of cohomogeneity one, and show that $M$ must be an open subset of one of our examples, for which $\gam$ can be in either $\CP^1$ or $\CH^1$ and have arbitrary 
constant curvature.  As an illustration of how much stronger the constant $|A|$ assumption is, 
in Section \ref{normA} we compute $|A|$ for the range of our examples, of course finding that 
it can be constant only for a ruled hypersurface in $\CH^n$ built on a curve with this specific value for $\kappa_g$.
(Our hypothesis that $M$ have cohomogeneity one excludes the Lohnherr hypersurface.)
\end{remark}
} \section{General Background and Terminology}

\label{background} 
We recall the canonical construction of $\cpn$ and $\chn$ (see \cite{nrsurvey}, pp.235-236). 
The complex space form $\cpn$ (resp. $\chn$) is based on an inner product on $\C^{n+1}$ that is positive definite (resp. indefinite).  In fact, as a point set, $\chn$ is the subset of $\cpn$ determined by $\<\bz, \bz \> <0$ for $\bz \in \C^{n+1}$.  In either case, the inner product is denoted by $\<\ ,\ \>$.  

When $\<\ ,\ \>$ is indefinite, we will also make use of the subset of $\cpn$ determined by $\<\bz, \bz \> >0$.  This is an indefinite K{\"a}hler manifold $\cpn_1=\cpn \mysetminus \overline{\chn}$ of constant positive holomorphic curvature $-4c>0$. The set $\overline{\chn}$ is determined by $\<\bz, \bz \> \le 0$.

For an element $T$ of the special unitary group 
\revise{$SU(n+1)$} (resp. $SU(n,1)$), let $\widetilde T(\pi \bz) = \pi (T\bz)$.  Then $\widetilde T$ is well-defined and is an isometry of $\cpn$ (resp. $\chn$).  When discussing hypersurfaces, the set of all such $\widetilde T$ is called the {\it ambient isometry group} since it is the group of isometries of the ambient space $\cpn$ (resp. $\chn$).  In the indefinite case, $\widetilde T$ also preserves $\cpn_1$ and the {\it ideal boundary} of $\chn$ which is $\{\pi \bz| \<\bz,\bz\>=0\}$.

For a hypersurface $M^{2 n - 1}$ in a complex space form $\cpn$ or $\chn$, the
structure vector is $W = - J \xi$ where $\xi$ is a locally-defined unit
normal. The $(2n-2)$-dimensional distribution $W^\perp$ is called the
holomorphic distribution.  
The shape operator $A$ is defined by
$$\nt_X \xi = - A X$$
where $\nt$ is the Levi-Civita connection of the ambient space. 
The tangential part of the complex structure $J$
is denoted by $\varphi$.  It annihilates $W$, acts like an almost complex
structure on $W^\perp$, and is related to the shape operator by the formula
\begin{equation}
\nonumber \na_X W=\varphi AX.
\end{equation}
The Gauss
equation expresses the curvature tensor of $M$ in terms of $A$ and $\varphi$ as
follows:
$$R(X,Y)=AX\wedge AY+c\(X\wedge Y+\varphi X\wedge \varphi Y +2 \<X,\varphi Y\>
\varphi\),$$ where $4c$ is the holomorphic curvature of the complex space form.
For convenience, we also use the positive constant $r$ satisfying 
$c = \pm 1/r^2$.

\begin{defn} If $W$ is a principal vector everywhere
(i.e., $A W = \a W$ for a smooth function $\a$), we say that $M$ is a {\it Hopf hypersurface}.
On the other hand, $M^{2n-1}$ is a {\it ruled hypersurface} if its shape operator
$A$ satisfies $A W^\perp \subseteq {\text{span} \ } W$.  
\end{defn}
It is not difficult to show that these are non-intersecting classes -- no Hopf hypersurface is ruled and no ruled hypersurface is Hopf.

For a ruled hypersurface, the range of $A$  lies in the span of $\{AW, W\}$, and so $\text{rank\ } A \le 2$ everywhere. 
We let \[\a = \<AW, W\>, \qquad \b = |AW - \a W|. \]
We note that $\b$ is continuous and globally defined on $\Mh$,
although it may fail to be smooth at points where it vanishes.
Every ruled hypersurface has a point where $\b > 0$ (otherwise, it would be Hopf).  
Such a point has a neighborhood where $\b$ remains positive and we can write $AW = \a W + \b U$ for a uniquely determined unit vector
field $U$.  The pair $\{W, U\}$ is orthonormal and  $AZ = 0$ for all $Z$ orthogonal to $W$ and $U$.

Note that $\b$ cannot vanish on any open subset of a ruled hypersurface.  However, it can vanish at isolated points and even on two-dimensional submanifolds.

\subsubsection* {Notational Clarification} 

To avoid confusion with terminology used by other sources, we establish the following:
\begin{itemize}
\item
For a curve with unit tangent vector $T$ in a Riemannian manifold with Levi-Civita connection $\nabla$, the {\it geodesic curvature} $\kappa_g$ is $|\nabla_T T|$.  Thus, geodesics are characterized by the condition $\kappa_g =0$.
\item
Let $(\Mh,G)$ be a pair consisting of a hypersurface $\Mh$ in a Riemannian manifold $\mt$ and a group $G$ of isometries of $\mt$ that preserves $M$. The action is said to be of {\it cohomogeneity $m$} if $m$ is the codimension in $M$ of an orbit $G\cdot\ppt$ of maximum dimension.  If $m=0$, then $\Mh$ is homogeneous, i.e., for any two points ${\xpt}_1$ and ${\xpt}_2$ in $\Mh$, there is a $g \in G$ such that $g {\xpt}_1= {\xpt}_2$.  If $m=1$, the orbits of codimension 1 are dense in $\Mh$.
\item
When $\mt = \cpn$, the homogeneous hypersurfaces are all Hopf and belong to a specific list (see Takagi's list, \cite{crbook}, p.350).
When $\mt = \chn$, there is an analogous list of homogeneous Hopf hypersurfaces (see Montiel's list, \cite{crbook} p.352), but there are also non-Hopf homogeneous hypersurfaces which have been classified by Berndt and Tamaru \cite{berndttamaru}.  Among the non-Hopf homogeneous hypersurfaces, however, only one is ruled, \revise{namely the Lohnherr hypersurface.}
\item
Let $\Mh$ be a ruled hypersurface of cohomogeneity $m$ in $\cpn$ or $\chn$. We say that $\Mh$ satisfies the {\it transversality condition} at a point $\ppt$ if  the tangent space to the orbit through $\ppt$ and the tangent space to the ruling through $\ppt$ are transverse subspaces of $T_\ppt \Mh$.  We also say the the group {\it acts transversely} at $\ppt$ in this case. 

\end{itemize}
 
\section{Characterization Theorems}\label{charT}

Here are our results for $n=2$:

\begin{theorem} \label{main-theorem}
Let $\Mh$ be a real-analytic ruled hypersurface of cohomogeneity one in a non-flat complex space form $\Ma$ (either $\cptwo$ or $\chtwo$) satisfying the transversality condition at every point.
Suppose that $\beta^2 +c$ is not identically zero on $\Mh$.  (Of course, if $c>0$ this condition is vacuous.)
Then 
\begin{enumerate}[label={\rm (\roman*)},ref={(\roman*)}]
\item $\beta^2+c$ is nonvanishing on a dense  open subset $\setU \subset \Mh$, and cannot change sign on $\Mh$;
\item there exists a complex projective line $\ell$ 
and a curve $\gam$ in $\ell$ 
such that each ruling of $\Mh$ is an open subset of a complex projective line $\ell_{\qpt}^{\perp}$ that passes through a
point $\qpt \in\gam$ and is perpendicular to $\ell$ at $\qpt$;
\item $\gam$ has constant geodesic curvature;
\end{enumerate}
Note that the line $\ell$ must be the projectivization of a subspace of $\C^3$ to which the inner product 
$\<\ , \  \>$ restricts to be nondegenerate.
When $\Ma = \chtwo$ we furthermore have:
\begin{enumerate}[resume,label={\rm (\roman*)},ref={(\roman*)}]
\item if $\beta^2+c < 0$ on $\setU$ then $\ell$ intersects $\chtwo$ in a totally geodesic $\chone$ containing $\gam$;
\item if $\beta^2+c > 0$ on $\setU$ then $\ell$ is totally geodesic in the indefinite K\"ahler manifold $\cptwo_1$ and does not intersect $\chtwo$.
\end{enumerate}
\end{theorem}

\begin{corstar}
Let $\Mh$ be a real-analytic ruled hypersurface of cohomogeneity one in $\cptwo$ or $\chtwo$ satisfying the transversality condition \revisemorestill{and such that $\beta^2+c$ is not identically zero}.
Then $\Mh$ is congruent to 
one of the ruled hypersurfaces constructed below in \S\ref{cp2Examples} and \S\ref{ch2Examples}.  
Conversely, every example occurring in those constructions is 
a cohomogeneity-one ruled hypersurface.   
\end {corstar}

 \section{Subspaces of $\C^3$} \label{subspaces}

For the ruled hypersurface construction we will describe in the next section, 
we need to choose an appropriate (complex) two-dimensional subspace $\scV \subset \C^{n+1}$ that will determine the space $M_0$ referred to in the introduction.  
If $\<\ , \  \>$ is definite on $\C^{n+1},$ its restriction to $\scV$ will, of course, be definite; however, when $\<\ , \  \>$ is indefinite, the restriction can be either definite or indefinite.   
(We exclude from our constructions the case where the restriction is degenerate.)

To simplify the exposition, we consider the case $n=2$.  We will also use the parameter
$\epsilon=\pm 1$, with $\epsilon=1$ when $\<\ , \  \>$ is positive definite and
$\epsilon=-1$ when $\<\ , \  \>$ is indefinite with signature $(-,-,+,+,+,+)$.
Let $\X' = \{\bz\in \C^3 | \<\bz, \bz\> = \epsilon r^2\}$ and
$\X = \pi \X'$ where $\pi$ denotes projectivization.  Note that $\X$ is $\cptwo$ or $\chtwo$ depending on the sign of $\epsilon$.

If $\epsilon=1$, then $\X = \cptwo = \pi \X' =\pi S^5(r) = \pi \C^3$. 
As a point set, $\pi \scV$ is a complex projective line $\cpone$ lying in $\cptwo$.
(Here, and in what follows, we abuse notation slightly by not explicitly noting that the origin is excluded from the domain of $\pi$.)   Its geometry is that of a totally geodesic submanifold with positive constant sectional curvature $4c$.  
There exists an orthonormal basis (over $\R$) for $\scV$ of the form $\{\e_0, \ri \e_0, \e_1, \ri \e_1\}$.  
(Note that multiplication by $\ri $ acts as an isometry in $\C^3$.) 
We can complete this to an orthonormal basis of $\C^3 =\R^6$ by adding a pair of the form $\{\e_2, \ri \e_2\}$.  

If $\epsilon=-1$, then $\X = \chtwo = \pi H^5_1(r) = \pi \X'$.  (Recall that $H^5_1(r)$ is the anti-de Sitter space defined by $\<\bz,\bz\>=-r^2$; see \cite{nrsurvey}, p.236.)   
We introduce a second parameter $\teps = \pm 1$ to indicate whether $\scV$ contains a timelike vector or not.  Specifically, let $\bv$ be any unit vector orthogonal to $\scV$ and let 
$\teps=\<\bv, \bv\>$.   

If $\teps = 1$ then the restriction of $\<\ , \  \>$ to $\scV$ is indefinite.  The complex projective line $\pi \scV$ is the union of three non-empty subsets, $\pi \scV^-$, $\pi \scV^0$ and $\pi \scV^+$. The first part $\pi \scV^-$ lies in $\chtwo$ as a totally geodesic $\chone$.  As a subset of $\pi \scV$ (which is topologically a 2-sphere $\cpone$), it is a 2-disk whose boundary is $\pi \scV^0$ and whose exterior is $\pi \scV^+$.  The latter may also be thought of as a totally geodesic complex hypersurface of the indefinite complex space form $\cptwo_1$ consisting of all $\pi \bz$ such that  $\<\bz, \bz\> = r^2$ (see \cite{barrosromero}, p.57).  The points of $\pi \scV^0$ are customarily 
referred to as {\em ideal points}, as they lie on the {\em ideal boundary} of 
$\chtwo$.

In this case, we let $\e_0$ be a unit timelike vector in $\scV$. Then $\e_0$ and $\ri \e_0$ span a (real) 2-dimensional subspace of $\C^3$ on whose orthogonal complement the restriction 
of $\<\ ,\ \>$ is positive definite.  (This is because the maximal real dimension of any totally isotropic subspace is 2.)    Let $\e_1$ be any spacelike unit vector in $\scV$ orthogonal to $\e_0$ and ${\mathrm i} \e_0$.  Then $\{\e_0, \e_1\}$ is a basis for $\scV$ over $\C$ and   
$\{\e_0,\ri \e_0, \e_1,\ri  \e_1\}$ is a basis for $\scV$ over $\R$.  Finally, we can choose a unit vector $\e_2$ so that $\{\e_0,{\mathrm i}   \e_0, \e_1,{\mathrm i}   \e_1, \e_2, {\mathrm i}  \e_2\}$ constitutes an orthonormal basis for $\R^6$.

If $\teps=-1$ then the restriction of $\<\ , \  \>$ to $\scV$ is definite.  Thus, $\pi\scV^0$ and $\pi\scV^-$ are empty and $\pi\scV=\pi\scV^+$ is contained in $\cptwo_1$.   We choose an orthonormal basis for $\scV$ of the form $\{\e_1, \ri \e_1, \e_2, \ri \e_2\}$, but in this case we
complete it by adding a pair $\{ \e_0, \ri \e_0\}$ of timelike unit vectors orthogonal to $\scV$.

To summarize, we will consider three cases:
\begin{enumerate}
\item  $\epsilon = 1$ (hence $\teps=1$ also) and $\scV$ is spanned over $\C$ by $\{\e_0, \e_1\}$;
\item  $\epsilon = -1$, $\teps = 1$, and $\scV$ is spanned by  $\{\e_0, \e_1\}$;
\item  $\epsilon = -1$, $\teps = -1$, and $\scV$ is spanned by $\{\e_1, \e_2\}$.
\end {enumerate} 
In all cases $\{\e_0, \ri  \e_0, \e_1, \ri \e_1, \e_2, \ri  \e_2\}$ is an orthonormal set, with $\<\e_0, \e_0\> = \epsilon$, $\<\e_1, \e_1\> =  \<\e_2, \e_2\> = 1$.

\section{Ruled hypersurfaces determined by plane curves} \label{construction}

In this section, we set up the mechanism for constructing a simple class of ruled hypersurfaces
in $\chtwo$ and $\cptwo$.  Much of the development can be handled
simultaneously for all cases using the parameters $\epsilon,\teps$ defined in the previous section. 
As above, we let $\scV \subset\C^3$ be a subspace to which $\<\ , \  \>$ restricts to be nondegenerate, and let $\bv$ be a vector orthogonal to $\scV$ with $\< \bv, \bv\>=\teps$.  (In terms of the bases chosen above, either $\bv= \e_0$ or $\bv=\e_2$.)

Let $\bdelta: I \mapsto \scV$ be a curve with $\<\bdelta(s), \bdelta(s) \> = \epsilon \teps r^2$ for all $s$ in some real interval $I$. 
Through each point of the curve we take the unique complex 1-dimensional subspace that is orthogonal to $\scV$; the union of these is a 3-dimensional submanifold $\Sigma$ of $\C^3$.  The 5-dimensional complex cone over $\Sigma$ intersects $\X'$ in a hypersurface $M'$ whose projection $\pi M'$ in $\X$ will be our desired hypersurface $M$.

We begin by parameterizing $M'$; specifically, let
\begin{equation}\label{basic-z}
 \bz = f(\theta, s, w) = e^{i \theta} \s (\bdelta(s) + w r \bv)
\end{equation}
where $\theta$ and $s$ are real parameters, $w$ is complex and $\s$ is a positive function chosen so that $\bz$ will lie in $\X' = \{\bz|\<\bz,\bz\>=\epsilon r^2\}$.
This will require
\begin{equation}\label{def-sigma}
\teps \s^2 (1 + \epsilon |w|^2) = 1,
\end{equation}
which determines $\sigma$ as a function of $w$.  It is useful to note that this relation implies 
$$\epsilon \s^2 |w|^2 =-(\s^2 -\teps).$$   

Next, we will construct an orthogonal basis for the tangent space to $M'$ by computing
derivatives of $\bz$.
Expressing $w$ in terms of its real and imaginary parts by $ w =w_1 + \ri w_2 $, we let
\begin{equation}\label{Rformula}
\bR= w_1 \frac{\partial \bz}{\partial w_1} + w_2 \frac{\partial \bz}{\partial w_2} = \teps(\s^2-\teps)\bz + e^{i \theta}r \s w \bv.
\end{equation}
We then compute $\<\bR, \bR\> = -\epsilon\s^2(\s^2-\teps)r^2 = \s^4 |w|^2 r^2$ and note that
 $\<\bR, \bz\> = \<\bR,\ri \bz\> =0.$ 
Furthermore, 
$\bz_\theta = \ri\bz$ and $\bz_s =  e^{\ri\theta} \s \bdelta_s$,
so that $\<\bz_s, \ri\bz\> = \sigma^2 \<\bdelta_s, \ri \bdelta\>$.
It is also clear that $\<\bz_s, \bR\> = 0$.
Thus,  $\{\bz, \ri\bz, \bR, \ri \bR\}$ is an orthogonal set, and if $w\ne 0$ it consists of nonzero vectors.
We will next check that the remaining derivatives of $\bz$ are in the span of $\{\ri \bz, \bR, \ri  \bR\}$; doing this
will require expressing $w$ in terms of real coordinates adapted to each case.

\subsubsection* {For $\cptwo$ (Case 1, $\epsilon=\teps=1$)}
Since $\s^2(1+|w|^2)=1$, we have $0<\s\le 1$.  
On the open set where $w\ne 0$, we introduce local real coordinates $t,u$ such that $w =e^{\ri t} \tan u$ and $\s=\cos u$, where $0 < u < \pi/2$.  In terms of these coordinates,
$$
\bz_u = \frac{1}{\sin u \cos u} \bR =r \frac{\bR}{|\bR|} $$
and so, using \eqref{Rformula} to express $\bv$ in terms of $\bR$ and $\bz$,
\begin{equation}
\bz_t = e^{i \theta} r \s w\ri  \bv = \ri \bR - \teps(\s^2 -\teps)\ri\bz = \ri\bR - (\s^2 -1)\ri\bz. \label{1zt-equation}
\end{equation}

\subsubsection* {For $\chtwo$ (Case 2, $\epsilon=-1$, $\teps=1$)}
Since $\s^2(1-|w|^2)=1$, we have $1\le \s < \infty $.  
On the open set where $w \ne 0$, introduce local coordinates $t,u$ such that $w = e^{\ri t}\tanh u$ and $\s = \cosh u$, where $u>0$.  In terms of these,
$$
\bz_u = \frac{1}{\sinh u \cosh u} \bR 
= r \frac{\bR}{|\bR|}
$$
and $\bz_t$ is given by the same formula as \eqref{1zt-equation}.

\subsubsection* {For $\chtwo$ (Case 3, $\epsilon=\teps=-1$)}

In this case, our curve $\bdelta$ lies in the span of $\{\e_1, \e_2\}$ and $\bv = \e_0$.   
Then $\scV$ intersects $S^5_2(r) = \{ \bz \in \C^3  | \<\bz, \bz\> = r^2\}$ in a 3-sphere $S^3$ of constant positive curvature $1/r^2$ and $\pi S^3$ inherits the geometry of a $\cpone$ with constant positive curvature $4/r^2 = - 4c$.  
The range of our curve $\bdelta$ lies in $S^3$ and $\pi \circ \bdelta$ will not intersect $\chtwo$.  
However, the analogous construction of perpendicular rulings can still be made, leading to a ruled hypersurface in $\chtwo$.

Since $\s^2(1-|w|^2)=-1$, then $|w|>1$, so the orthogonal set $\{\bz, \ri\bz, \bR, \ri \bR\}$ automatically consists of nonzero vectors in this case.
We choose local coordinates $t,u$ so that $w=e^{\ri t} \coth u$ and $\sigma=\sinh u$ where $u>0$.  In terms of these,
$$
\bz_u = -\frac{1}{\sinh u \cosh u} \bR 
= - r \frac{\bR}{|\bR|}
$$
and
\begin{equation}
\bz_t = e^{i \theta} r \s w \ri \bv =\ri \bR - \teps(\s^2 -\teps){\mathrm i}\bz  = \ri \bR + (\s^2 +1)\ri\bz. \label{3zt-equation}
\end{equation}

\subsection* {For all three cases}
In all cases, $\bz_u = \teps r \dfrac{\bR}{|\bR|}$ and $\bz_t$ is in the span of $\ri \bz$ and $\ri \bR$. 
To complete an orthogonal basis, we will compute the part of $\bz_s$ orthogonal to the span of $\{ \ri \bz, \bR, \ri \bR\}$.  This vector is
\begin{equation}
  \be = \bz_s -\epsilon b (\s^2 \ri\bz -\teps \ri\bR) \label{edef}
\end{equation}
where $b = \<\bdelta_s, \ri \bdelta\>/r^2$.
When $\epsilon = -1$, the orthogonal set $\{\bz,\ri \bz, \bR, \ri\bR\}$ has two timelike and two spacelike vectors, so $\<\be, \be\>$ is nonnegative no matter what the sign of $\epsilon$ is.
Using \eqref{Rformula}, we rewrite the formula for $\be$ as
\begin{equation}\label{secondeformula}
\be = e^{\ri\theta } \s (\bdelta_s - \teps\epsilon b \ri \bdelta).
\end{equation}
Letting  $\rho = |\be|$, we compute
\begin{equation}\label{rho=sigma}
\rho^2 = \s^2(\<\bdelta_s, \bdelta_s\> -\teps \epsilon b^2 r^2).
\end{equation}

\begin{prop}
The map $f(\theta, s, w) = e^{i \theta} \s (\bdelta(s) + w r \bv)$ is an
immersion at $(\theta, s, w)$ if and only if $\rho\ne 0$.  This holds even where $w = 0$.
\end{prop}
\begin{proof}
Assuming that $w \ne 0$, consider any point where $\rho \ne 0$.  Then  $\{\bz, \ri\bz, \bR, \ri\bR, \be, \ri\be\}$
is an orthogonal set of nonzero vectors in $\C^3$.
The set $\{\bz_\theta, \bz_s, \bz_u, \bz_t\}$ is linearly independent since it is
related to $\{\ri\bz, \bR, \ri\bR, \be\}$ by the nonsingular matrix
$$
\begin{bmatrix}
\quad 1 &\quad \epsilon\s^2 b &\quad 0 &\quad -\teps(\s^2-\teps) \\
\quad 0 &\quad 0 &\quad \teps/\s^2 |w| &\quad 0 \\
                \quad 0 &\quad -\teps\epsilon  b &\quad 0 &\quad 1 \\
                         \quad 0&\quad 1 &\quad 0&\quad 0
\end{bmatrix}.
$$
Thus $f$ must be an immersion at the point in question.
\par
Consider now a point where $\rho = 0$. Using
\eqref{1zt-equation} or \eqref{3zt-equation} and \eqref{edef}, we can express
$\bz_s$ as a linear combination of $\bz_\theta$ and $\bz_t$, so that 
$f$ is not an immersion at that point. 

We now consider a point where $w=0$.  Note that $\teps =1$ and $\bv=\e_2$ in this case.  
Also, $\s=1$ at the point in question so that 
$\bz_\theta = e^{\ri\theta} \ri\bdelta$, $\bz_s = e^{\ri\theta}\bdelta_s$, $\bz_{w_1}= e^{\ri\theta}r \e_2$, and $\bz_{w_2}= e^{\ri\theta}r \ri\e_2$ there.  Thus, $f$ will be an immersion if and only if $\{\bdelta_s, \ri\bdelta, \e_2, \ri\e_2\}$ is a linearly independent set.  This is equivalent to linear independence of $\{\bdelta_s, \ri\bdelta\}$.  Now $\bzeta = \bdelta_s-\teps \epsilon b \ri\bdelta$ is orthogonal to $\bdelta$ and $\ri\bdelta$ and thus $f$ is an immersion if and only if $\bzeta \ne 0$. 
From \eqref{secondeformula} it is clear that this is equivalent to $\rho \ne 0$.
\end{proof}
\par

We resume our work in an open set where $\rho\ne 0$ and the $(u, t)$ parameterization is valid.  
We will define an orthonormal basis for the tangent space to $M'$ and compute the shape operator of $M'$ as a hypersurface in $\X'$.
This, in turn, will allow us to compute the shape operator of $M$.

We have a hypersurface in $\X'$ with unit normal $\bxi' = \ri \be / \rho$
and tangent space spanned by $\{\ri\bz, \be, \ri\bR, \bR\}$.  
Let $\bU$ be the unit vector field $\be/\rho = - \ri \bxi'$.  Then $\xi = \pi_* \bxi'$ and $W = \pi_* \bU = - J \xi$ will be the
 unit normal and structure vector fields for the hypersurface $M$ in $\X$.  Further, let
\begin{equation}
\boldE'_1 = \frac{\ri\bR}{ |\bR|}, \qquad
\boldE'_2 = \frac{\bR}{|\bR|}.
\end{equation}
(Recall that $|\bR| = \s^2 |w| r$.)
It is easy to check that $(u, t)$ satisfy
\begin{align}
\bz_u &= \teps r \boldE'_2. 
\end{align}

In computing the shape operator $A'$ of $M'$ in $\X'$, we make use of the fact
that $\X'$ is totally umbilic in $\C^3$ so that the partial derivative of $\bxi'$ with
respect to any coordinate coincides with its covariant derivative on
$\X'$.  We begin by computing explicitly that $A'(\ri\bz) = \bU$, even though this
fact must be true for any $S^1$-invariant hypersurface (see \cite{nrsurvey}, Lemma 1.7).
$$
A'(\ri\bz) = - \nt_{\ri\bz}\bxi' = -\frac{\partial \bxi'}{\partial \theta} = -\ri \bxi' = \bU.
$$
Continuing with our computation of $A'$, we have
\begin{equation}
A'\boldE'_2 = \frac{\teps}{r} A'\bz_u = -  \frac{\teps}{r} \frac{\partial \bxi'}{\partial u} =0
\end{equation}
since
\begin{equation}
\frac{\be}{\rho} =  \frac{\be}{\s}(\<\bdelta_s, \bdelta_s\> - \teps\epsilon b^2 r^2)^{-\frac{1}{2}}
\end{equation}
depends on $\theta$ and $s$ but not on $t$ or $u$.  Also,
\begin{align}
A'\boldE'_1 &= \frac{1}{r\s^2|w|} A' (\ri\bR) = \frac{1}{r\s^2|w|}(A'\bz_t + \teps(\s^2-\teps)A'(\ri\bz))\\
&= \frac{1}{r\s^2|w|}\left(-\frac{\partial}{\partial t}\left(\frac{\ri\be}{\rho}\right)\right) -\frac{\teps \epsilon }{r} |w| \ A'(\ri\bz)\\
&= - \frac{\teps\epsilon}{r} |w| \ \bU.
\end{align}

It remains to compute $A'\bU = \frac{1}{\rho} A' \be = \frac{1}{\rho} \left(A' \bz_s  -\epsilon b(\s^2  A'(\ri\bz) -\teps A'(\ri\bR))\right)$.
From the previous steps, it is easy to see that $\s^2 A'(\ri\bz) -\teps  A'(\ri \bR)\ =  \teps \bU$.  Now
$$
A'\bz_s = - \nt_{\bz_s}\bxi'  = -\frac{\partial}{\partial s}\left(\frac{\ri \be}{\rho}\right)
= \frac{\rho_s}{\rho^2}\ \ri\be -\frac{1}{\rho}(\ri \be)_s.
$$
The first term of this expression reduces to
$$
\frac {\< \bdelta_{ss}, \bdelta_s\> - \teps\epsilon b \<\bdelta_{ss}, \ri\bdelta\>}
      {\<\bdelta_s, \bdelta_s\> -\teps\epsilon  b^2 r^2}\ \bxi'.
      $$
\noindent
Differentiating (\ref{edef})
gives
\begin{equation}
\be_s = \bz_{ss} -\epsilon \frac { \<\bdelta_{ss}, \ri\bdelta\>}{r^2}\s^2 \ri\bz
+ \teps\epsilon  \frac { \<\bdelta_{ss}, \ri\bdelta\>}{r^2} \ri\bR -\teps\epsilon b \,\ri \bz_s
\end{equation}
since $\bR_s = \teps(\s^2-\teps)\bz_s$.

On the other hand, if we express $\bz_{ss} = e^{i \theta} \s \bdelta_{ss}$ in terms of the orthogonal frame
 $\{\bz,$ $\ri\bz$, $\be$, $\ri\be$, $\bR$, $\ri \bR\}$, we get
\revise{
\begin{equation*}
\begin{split}
r^2 \bz_{ss} =&
 \epsilon \s^2 \<\bdelta_{ss}, \bdelta\> \bz 
    +r^2\frac {\< \bdelta_{ss}, \bdelta_s\> -\teps\epsilon b \<\bdelta_{ss}, \ri\bdelta\>}
      {\<\bdelta_s, \bdelta_s\> - \teps\epsilon b^2 r^2}\ \be
          -\teps\epsilon \<\bdelta_{ss}, \bdelta\> \bR \\
&+ \epsilon \s^2 \<\bdelta_{ss}, \ri\bdelta\> \ri\bz
    +r^2\frac {\< \bdelta_{ss}, \ri\bdelta_s\> + \teps\epsilon b \<\bdelta_{ss}, \bdelta\>}
      {\<\bdelta_s, \bdelta_s\>\ -\teps\epsilon  b^2 r^2}\ \ri\be
           -\teps\epsilon \< \bdelta_{ss}, \ri\bdelta\> \ri\bR 
\end{split}
\end{equation*}
}
from which we compute
\revise{
\begin{equation*}
\begin{split}
\be_s =& \epsilon \s^2 \frac {\<\bdelta_{ss}, \bdelta\>}{r^2}\ \bz
+ \frac {\< \bdelta_{ss}, \bdelta_s\> - \teps\epsilon 
b \<\bdelta_{ss}, \ri\bdelta\>} {\<\bdelta_s, \bdelta_s\> - \teps\epsilon b^2 r^2}\ \be
 - \teps\epsilon  \frac {\<\bdelta_{ss}, \bdelta\>}{r^2} \bR
\\&
+ \frac {\< \bdelta_{ss}, \ri\bdelta_s\> + \teps\epsilon b \<\bdelta_{ss}, \bdelta\>}{\<\bdelta_s, \bdelta_s\> -\teps\epsilon b^2 r^2} \ri\be
 - \teps\epsilon b\, \ri\bz_s,\\
\end{split}
\end{equation*}
}
so that, using $\< \bdelta_{ss}, \bdelta\> = - \<\bdelta_s, \bdelta_s\>$, we have
\begin{equation*}
A'\bz_s =
-\frac{1}{\rho} \left(- \epsilon \s^2 \frac {\<\bdelta_s, \bdelta_s\>}{r^2}\ {\mathrm i}\bz
-\frac {\< \bdelta_{ss},\ri \bdelta_s\> - \teps\epsilon b \<\bdelta_s, \bdelta_s\>}{\<\bdelta_s, \bdelta_s\> - \teps\epsilon b^2 r^2}\ \be + \teps\epsilon \frac {\<\bdelta_s, \bdelta_s\>}{r^2} \ri\bR
+\teps\epsilon  b \ \bz_s\right).
\end{equation*}
\noindent
Thus,
\begin{equation*}
A'\bU = \epsilon \frac{\ri\bz}{r^2} +
\frac{1}{\rho} \left(
\frac {\< \bdelta_{ss}, \ri\bdelta_s\> 
- \teps\epsilon b \<\bdelta_s, \bdelta_s\>}{\<\bdelta_s, \bdelta_s\> - \teps\epsilon b^2 r^2}\ -2\teps\epsilon b\right) \bU
-\teps\epsilon \frac{|w|}{r} \  \boldE'_1.
\end{equation*}

\par
\me
We have shown that
with respect to the orthonormal basis 
$(\frac{\ri\bz}{r}, \bU, \boldE_1', \boldE_2')$,
$$
 A'= \begin{bmatrix}\

0 & \frac{\epsilon}{r} & 0 & 0 \\
\frac{1}{r} & \a &\b & 0 \\
                0 &\b &0 & 0 \\
                          0& 0 &0& 0 \\
\end{bmatrix}
$$
where $ \b =-\teps\epsilon\frac{|w|}{r}  $ and
\begin{equation}\label{alpha-definition}
\alpha = \frac{1}{\rho} \left(
\frac {\< \bdelta_{ss}, \ri\bdelta_s\> -\teps\epsilon  b \<\bdelta_s, \bdelta_s\>}{\<\bdelta_s, \bdelta_s\> - \teps\epsilon b^2 r^2}\ -2\teps\epsilon b\right).
\end{equation}

Following the methods of \cite{nrsurvey}, we see that the projection $M$ of $M'$ to
$\X$ is a hypersurface with unit normal
$\xi = \pi_*\bxi'$ and structure vector $W = \pi_*\bU$ as indicated above.  Let $E_1 = \pi_*\boldE_1'$ and $E_2 = \pi_*\boldE_2'$.
With respect to the basis $(W, E_1, E_2)$, the shape operator
of $M$ is
\begin{equation}
 A = \begin{bmatrix}
 \a &\b& 0 \\
                \b &0 & 0 \\
                           0 &0& 0 \\
\end{bmatrix}.
\label{ruled-shape}
\end{equation}

For any hypersurface in which $A$ takes this form,
the characteristic polynomial of $A$ is $-t(t^2 -\a t -\b^2)$,
so that $ \frac{1}{2} (\a + \sqrt{\a^2 + 4 \b^2})$ and
$\frac{1}{2} (\a - \sqrt{\a^2 + 4 \b^2})$, along with 0,
are continuous principal curvature functions. In our constuction, $\pi \circ \bdelta$ is precisely the set of points where $\b=0$. 
This is called the {\it spine} of the hypersurface (see Goldman  \cite{goldman}, Chapter 5, for more about spines).  Note that $\s =1$ on the spine. 
Except for points on the spine, the principal curvature functions in our construction are smooth, as is the frame $\{W, E_1, E_2\}$.

\begin{remark}
It is not difficult to see how the analogous construction can be made in higher dimensions.  The single complex dimension determined by $\bv$ must be replaced by the $\C^{n-1}$ orthogonal to $\scV$.  The value of $\a$ depends only on calculations performed in the complex 2-dimensional space $\scV$ (see \eqref{alpha-definition}) and the value of $\b$ depends only on the distance from the spine as measured by $|w|$.  The shape operator is zero on all the ``extra" directions.  Extending the special cohomogeneity-one examples (see next section) to higher dimensions is straightforward.  \revise{Obtaining the characterization theorems by the method of moving frames (see \S \ref{Moving-Frames}) is also straightforward}.
\end{remark}

 \section {Cohomogeneity-one examples} \label{coho-1-Construction}
In this section, we construct the simplest examples -- ruled hypersurfaces based on curves of constant geodesic curvature.

\subsection{The curve $\pi \circ \bdelta$}\label{curve}

We look more closely at the curve over which our hypersurfaces will be built.  Specifically, we compute its curvature in terms of the parameters occurring in the previous section.

First recall the 2-dimensional complex subspace
$\scV$ of $\C^3$ and the unit vector $\bv$ orthogonal to it.
To accommodate the case where the curve lies outside the ambient space of the hypersurface, we define the 3-dimensional submanifold 
$$\X_0' = \{\bz \in \C^3| \<\bz,\bz\> = \teps \epsilon r^2, \<\bz,\bv\> =\<\bz, {\mathrm i} \bv\> =0\} $$
which contains the image of $\bdelta$.
Since we are temporarily working with the case $w=0$, we write
$$ \bz = f_0(\theta, s) = e^{{\mathrm i}  \theta} \bdelta(s) $$ and take
\begin{equation}
  \be = \bz_s - \teps\epsilon b {\mathrm i} \bz.
\end{equation}
As before, 
 $\<\be, \bz\> = \<\be, {\mathrm i} \bz\>  =0$ and $\rho = |\be|$ satisfies
 $$\rho^2 = \<\bdelta_s, \bdelta_s\> -\teps\epsilon b^2 r^2.$$
 Provided that $\rho \ne 0$, $f_0$ defines a surface $M'_0$ in  $\X_0'$ with unit normal $\bxi' = {\mathrm i}  \be / \rho$
and tangent space spanned by $\{ {\mathrm i} \bz, \bU\}$ 
where $\bU= \dfrac{1}{\rho}\be$. 

Now $\X_0 = \pi \X'_0$ is either $\cpone$ or $\chone$ and $ \pi \circ \bdelta$ is a curve in $\X_0$ which we denote by $M_0$.  
Then $\xi = \pi_* \bxi'$ and $W = \pi_* \bU = - J \xi$ will be
 unit normal and tangent vector fields for the curve 
$M_0$ in $\X_0$.

The shape operator of $M_0'$ in $\X_0'$ is just the restriction of $A'$ which we computed previously.  In particular, $A'( {\mathrm i} \bz) = \bU$, 
and $A'\bU = \dfrac{\epsilon} {r^2} {\mathrm i} \bz + \a \bU$
where, recalling \eqref{alpha-definition},
$$
\alpha = \frac{1}{\rho} \left(
\frac {\< \bdelta_{ss}, {\mathrm i} \bdelta_s\> - \teps\epsilon b \<\bdelta_s, \bdelta_s\>}{\<\bdelta_s, \bdelta_s\> -\teps\epsilon b^2 r^2}\ -2\teps\epsilon b\right).
$$
\par
\me
We have shown that
with respect to the orthonormal basis $\{\frac{1}{r} {\mathrm i} \bz, \bU\}$,
$$
 A'= \begin{bmatrix}\
0 & \frac{\epsilon}{r} \\
\frac{1}{r} & \a 
\end{bmatrix}.
$$

Following the methods of \cite{nrsurvey} pp.240-242, we see that the projection $M_0$ of $M_0'$ to
$\X$ is a curve with unit normal
$\xi = \pi_*\bxi'$, and unit tangent vector $W = \pi_*\bU$ as indicated above.  Since $\chone$ (resp. $\cpone$) is totally geodesic in $\chtwo$ (resp. $\cptwo$), the covariant derivative there is obtained by restriction. It is easy to check that
$\nt_W W = \a \xi$ and $\nt_W \xi = -\a W$.  In other words, $M_0$ is a curve in $\X_0$ with geodesic curvature $|\a|$. 

\subsection {Building the Ruled Hypersurface}  \label{build}

Here we give an informal discussion which serves as motivation for our (more formal) construction of hypersurfaces based on curves of constant geodesic curvature.  Note that this discussion does not apply to Case 3 in Section \ref{construction} \revise{(``circles outside $\chtwo$")}  but the formalism for that case is analogous.

\subsubsection* {The idea}
The tangent space to $\X_0$ is $J$-invariant, and so is spanned at each point of $M_0$ by $W$ and $J W = \xi$, the unit normal to $M_0$.  Note that $W$ and $\xi$ are just the restrictions to $M_0$ of the vector fields of the same name occurring on the ruled hypersurface $M$.  In particular, through each point of $M_0$, the leaf of the distribution $W^{\perp}$ (on $M$) is contained in a unique complex projective line through that point.

Thus for any point $\qpt =\pi \bz$ on the curve $\gam = \pi \circ \bdelta$, the leaf of $W^\perp$ through $\qpt$ lies in the (totally geodesic) $\cpone$ (respectively, $\chone$) through $\qpt$ orthogonal to $W$ and $\xi$.  This tells us that the ruled hypersurface $M$ is obtained from the ``plane curve" $\gam$ by passing the unique perpendicular complex projective line through each of its points.

We obtain curves of constant geodesic curvature by choosing actions on $\cpone$ and $\chone$ whose orbits are the desired curves.  These actions extend trivially to $\cptwo$ and $\chtwo$.
Further, we can also find actions that rotate the leaves of $W^\perp$ and hence preserve the hypersurface $M$.  These two types of actions combine to make a 2-parameter group of transformations of $\cptwo$ (resp. $\chtwo$) whose orbits are of codimension one  in $M$ (see \S \ref{actions}).

More specifically, we replace $f(\theta, s, w)$ by 
$\bz=f(\theta, w) = e^{ {\mathrm i} \theta}\s (\bdelta(s) + w r \bv)$, in our general setup (holding $s$ fixed). The image of $f$ will be a 3-dimensional submanifold of the sphere $S^5(r)$ (resp. anti-de Sitter space $H^5_1$) and $\pi \circ f$ gives a 2-dimensional submanifold $L_0$ which is a punctured $\cpone$ (one point removed) in $\cptwo$ or a $\chone$ in $\chtwo$).  Applying the group action to $L_0$ traces out the hypersurface $M$.  The rulings of $M$ are $L_0$ and its images under the group action. 

One needs to verify that the submanifolds in question are well-defined.  When $w\ne 0$, the $(u,t)$ coordinate system may be used.  As before, we have
 $$
\bz_u = \teps r \frac{\bR}{|\bR|}
$$
and
\begin{equation}
\bz_t ={\mathrm i}  \bR - \teps(\s^2 -\teps){\mathrm i} \bz.
\end{equation}
It is easy to check that $\{\bz_\theta = {\mathrm i}  \bz, \bz_u, \bz_t\}$ is a linearly independent set so that $f$ is an immersion.  On the other hand, points where $w=0$ only arise in the case $\teps=1$.  Then $\bz_\theta = {\mathrm i} \bz$, $\bz_{w_1}= e^{ {\mathrm i} \theta}r \e_2$, and $\bz_{w_2}= e^{{\mathrm i} \theta}r {\mathrm i} \e_2$.  Clearly, these three vectors are linearly independent at such a point since $\bz$ lies in the span of $\{\e_0, \e_1\}$.
Thus, $f$ is an immersion at all points. 

In the next section, we will construct the examples explicitly.

\begin{remark}
The curves of constant geodesic curvature in $S^2=\cpone$ and $\rhtwo=\chone$
are orthogonal trajectories of pencils of lines occurring in projective and hyperbolic geometry.  Specifically, we observe the following:

\vspace{0.5cm}
In the real hyperbolic plane, there are three distinct kinds of pencils of lines: (see \cite{ryan})

1. pencils of intersecting lines (all lines through a given point $C$);

2. pencils of ultraparallels (all lines perpendicular to a given line $\ell$);

3. pencil of parallels (all lines directed toward a given ``ideal point" $\omega$).
\vspace{0.5cm}

\noindent
Correspondingly, there are three kinds of families of curves in which we are interested:

1. all circles centered at a point $C$;

2. a line $\ell$ and all its equidistant curves;

3. all horocycles ``passing through" an ideal point $\omega$.

\vspace{1.0cm}

In spherical geometry, pencils of the first two types coincide and there are no parallel pencils.  Thus, in our study of ruled hypersurfaces in $\cptwo$, the three cases collapse to one.  We will use the ``concentric circles" descriptions for these pencils.

\end{remark} 

\subsection{ Our examples in $\cptwo$}\label{cp2Examples} 
\bi\me\par

Note that the $\cpone$s that are the lines of $\cptwo$ will have constant curvature $4c = 4/r^2$.  Thus they are isometric to ordinary 2-spheres of radius $\frac{r}{2}$.  
Let
$$
 Y = \begin{bmatrix}
 {\mathrm i} & 0 \\
              0 &-{\mathrm i} & \\
\end{bmatrix}
$$
so that
$$\exp tY = \begin{bmatrix}
 e^{{\mathrm i}t} &\quad 0 \\
             0 &\quad e^{-{\mathrm i}t}  \\
\end{bmatrix}.
$$
These matrices represent transformations with respect to the basis $\{\e_0, \e_1\}$ (since $\e_2$ plays no role in this part of our construction).  
Let $\l$ be a nonzero real constant to be determined later.
Then $s \mapsto \exp (s\l Y)$ determines a 1-parameter
group of isometries of $\cpone$.  We get a 1-parameter group of isometries of $\cptwo$ by extending it to act as the identity on $\e_2$.
Let $a$ be a real number satisfying $0<a<\frac{\pi}{2}$ and take
$$
\bp = r (\cos\ a\ \e_0 + \sin\ a\ \e_1) \in S^3(r).
$$
Then
\begin{align}
{\mathrm i}\bp &= r{\mathrm i} (\cos\ a\ \e_0 + \sin\ a\ \e_1)\\
Y\bp &= r{\mathrm i} (\cos\ a\ \e_0 - \sin\ a\ \e_1).
\end{align}

Let $\bdelta(s) = \exp (s\l Y) \bp$ so that $\bdelta_s = \l (\exp (s\l Y))Y \bp$
and $\bdelta_{ss} = -\l^2 \bdelta.$  Then $\<\bdelta_s, \bdelta_s\> = \l^2r^2$,
$b = \l \cos 2a$, and $\<\bdelta_s, \bdelta_s\> - b^2 r^2 = r^2\l^2\sin^2 2a.$

We choose $\l = \frac{1}{r}\csc 2a$  so that $\rho = \s$ in \eqref{rho=sigma}.  We now apply the general construction of \S 
\ref{construction} to obtain a ruled hypersurface $M$ in $\cptwo$.   
Substituting in \eqref{alpha-definition}, we get
\begin{equation}
\a = -\frac{2}{r \s} \cot 2a
\end{equation}
and
\begin{equation}
\b = -\frac {|w|}{r}
\end{equation}
for this hypersurface.

\bi The spine of $M$ is a circle of radius $r a$ centered at $C=\pi \e_0$.  Each value of $a$ gives a different hypersurface.
The corresponding family of hypersurfaces
foliates the ``punctured'' $\cptwo$ which would result from the removal
of the $\cpone$ through $C$ and $\pi \e_2$.  For a point $\pi \bp$ on the spine, the extended ruling is the $\cpone$ perpendicular at $\pi \bp$ to the $\cpone$ through $C$ and $\pi \e_1$.
Thus, each ruling is a punctured $\cpone$.  As $|w|$ tends to $\infty$, $\pi \bz$ approaches but does not reach the point $\pi \e_2$ through which the corresponding the extended ruling passes.   The hypersurface $M$ is thus diffeomorphic to $S^1 \times \R^2$.  
Unlike the examples we will provide for $\chtwo$, the hypersurface $M$ is not complete.

Note that the curve $\gam$ is a curve of constant geodesic curvature $|\a|$, so that varying $a$ produces a family of concentric circles on the 2-sphere 
$\cpone$.  When $a=\frac{\pi}{4}$ this circle is a geodesic of the 2-sphere.  In this way, the circles are analogous to a pencil of  ``equidistant curves" as occur on $\chtwo$, where the geodesic plays the role of the base line (see next section)

As we have seen, the continuous principal curvature functions are $0$,  $\frac{1}{2} (\a + \sqrt{\a^2 + 4 \b^2})$ and
 $\frac{1}{2} (\a - \sqrt{\a^2 + 4 \b^2})$. 
In our case, 
$$
\a^2 + 4 \b^2 = \frac{4}{r^2} (|w|^2 + (1+|w|^2)\cot^2 2a)) .
$$
If $a=\frac{\pi}{4}$, then $\a$ vanishes identically and the principal curvatures are $0$ and $\pm \b$.  
The latter are not smooth at points where $\b=0$, namely on the spine $\gam$. 
For other values of $a$, $\a^2 + 4\b^2$ is nonvanishing and the principal curvature functions are smooth as is the frame $\{W, E_1, E_2\}$. 

Observe also that if we set $a = \frac{\pi}{2}$ in the above construction, we find that $\<\bdelta_s, \bdelta_s\> - b^2 r^2 =0$
so that $\rho = 0$ and the map $f$ is not an immersion.  As $a$ approaches $\frac{\pi}{2}$,  $|\a|$ tends to $\infty$.

\bi

\subsection{ Our examples in $\chtwo$} \label{ch2Examples}
 \bi\me\par
\subsubsection* {Hypersurfaces generated by equidistant curves}

Let
$$
 Y = \begin{bmatrix}
 0 & 1 \\
              1 &0& \\
\end{bmatrix}.
$$
so that
$$\exp tY = \begin{bmatrix}
  \cosh\ t &\quad \sinh\ t \\
             \sinh\ t &\quad \cosh\ t  \\
\end{bmatrix}
$$
Here, again, the matrices represent transformations with respect to the basis $\{\e_0, \e_1\}$. 
Let $\l$ be a nonzero real constant to be determined later.
Then $s \mapsto \exp (s\l Y)$ is a 1-parameter subgroup of $GL(2,\C)$
preserving $\<\ ,\ \>$ and thus induces a 1-parameter group of isometries
of $\chone$.  Note that $GL(2,\C)$ (respectively $\chone$)
lies in $GL(3,\C)$ (respectively $\chtwo$) in a canonical way, so that our group
may be regarded as consisting of isometries of $\chtwo$ that preserve $\chone$.

Let $a$ be an arbitrary real number and take
$$
\bp = r (\cosh\ a\ \e_0 + \sinh\ a\ {\mathrm i} \e_1) \in H^5_1(r).
$$
Then
\begin{align}
{\mathrm i}\bp &= r(\cosh\ a\ {\mathrm i} \e_0 - \sinh\ a\ \e_1)\\
Y\bp &= r(\sinh\ a\ {\mathrm i} \e_0 + \cosh\ a\ \e_1).
\end{align}

Let $\bdelta(s) = (\exp (s\l Y))\bp$, so that $\bdelta_s  = \l (\exp (s\l Y))Y \bp$
and $\bdelta_{ss} = \l^2 \bdelta.$  A routine calculation gives $\<\bdelta_s, \bdelta_s\> = \l^2r^2$,
 $b = -\l \sinh 2a$, and $\<\bdelta_s, \bdelta_s\> + b^2 r^2 = r^2\l^2\cosh^2 2a.$

We choose $\l = \frac{1}{r}\sech\ 2a$ so that $b= -\frac{1}{r} \tanh\ 2a$.
This gives $\<\bdelta_s, \bdelta_s\> + b^2 r^2 =1$ and hence $\rho = \sigma$ by \eqref{rho=sigma}.  
Because of \eqref{def-sigma} we have $\rho = \s =1$
when $w=0$ which will require $\bdelta$ to project to a unit speed curve $\g = \pi \circ \bdelta$ in $\chone$. 
This is because $\bdelta_s+b {\mathrm i} \bdelta$ is the component of $\bdelta_s$ orthogonal to the span of $\{\bdelta, {\mathrm i}\bdelta\}$ and thus its length is that of $\pi_*\bdelta_s$.

Finally, we compute $\alpha$
and find that
\begin{equation}
\a = -\frac{2}{r \s} \tanh 2a
\end{equation}
and
\begin{equation}
\b = \frac {|w|}{r}.
\end{equation}

When $a=0$, $\gam$ is the geodesic through $\pi \e_0$ in the direction of $\pi_* \e_1$ and $M$ is a
{\it bisector} as discussed by Goldman \cite {goldman}, p.153 (see also \cite{crbook}, p.446). Although there are many real
hyperbolic planes through $\g$, only one of them
is preserved by the complex structure (i.e., it is a complex hyperbolic line $\chone$).   For any pair of points $({\mathrm P}, {\mathrm Q})$ in this $\chone$ that are
symmetrically placed with respect to $\g$, the set of all
points in $\chtwo$ that are equidistant from ${\mathrm P}$ and ${\mathrm Q}$ is precisely $M$,
hence the name ``bisector".  Note that for a bisector, $\a$ is
identically zero.  The three principal curvatures are $0$, $\frac {1}{r}|w|$ and
$-\frac {1}{r} |w|$.  Thus bisectors are minimal hypersurfaces.  All three principal
curvatures are distinct, except on the spine where $A$
is identically zero.  Finally, we note that it is impossible
to define the principal curvature functions so as to be smooth in
any neighborhood of a point on the spine of a bisector.

For other values of $a$, $\g$ is the equidistant
curve at signed distance $ra$ from the geodesic described above.  The family of ruled hypersurfaces (as $a$ varies through all real
numbers) foliates $\chtwo$.
\par\me

When $a \ne 0$, our formulas for the principal curvatures extend to the spine and the principal curvature functions are smooth on all of $M$.  For points not on the spine, the three principal curvatures are distinct. On the spine, they are $\a$ (of multiplicity 1) and 0 (of multiplicity $2$).

We note that for the hypersurfaces defined in this section, all extended rulings pass through $\pi \e_2$.  Since this point is outside $\chtwo$ it does not interfere with the completeness of the hypersurface as happens in the $\cptwo$ case.

\bi
\subsubsection* {Hypersurfaces generated by circles}
\bi
For this case, we use the general formulas with $\teps = 1$ and $\epsilon = -1$.
Let
$$
 Y = \begin{bmatrix}
 {\mathrm i} & 0 \\
              0 &-{\mathrm i} & \\
\end{bmatrix}
$$
so that
$$\exp tY = \begin{bmatrix}
 e^{{\mathrm i}t} &\quad 0 \\
             0 &\quad e^{-{\mathrm i}t}  \\
\end{bmatrix}.
$$
Again, let $\l$ be a nonzero real constant to be determined later.
Then $s \mapsto \exp (s\l Y)$ gives a 1-parameter
group of isometries, as before.
Let $a$ be any positive number and take
$$
\bp = r (\cosh\ a\ \e_0 + \sinh\ a\ \e_1) \in H^5_1(r).
$$
Then
\begin{align}
{\mathrm i}\bp &= r{\mathrm i} (\cosh\ a\ \e_0 + \sinh\ a\ \e_1)\\
Y\bp &= r{\mathrm i} (\cosh\ a\  \e_0 - \sinh\ a\ \e_1)
\end{align}

Let $\bdelta(s) = (\exp (s\l Y))\bp$ so that $\bdelta_s  = \l (\exp (s\l Y))Y \bp$
and $\bdelta_{ss} = -\l^2 \bdelta.$  By a similar calculation, we get $\<\bdelta_s, \bdelta_s\> = -\l^2r^2$,
 $b = - \l \cosh\ 2a$, and $\<\bdelta_s, \bdelta_s\> + b^2 r^2 = r^2\l^2\sinh^2\ 2a.$

This time we let $\l = \frac{1}{r}\csch\ 2a$.
Note that this is possible only because $a \ne 0$.
Finally,
\begin{equation}
\a = -\frac{2}{r \s} \coth\ 2a
\end{equation}
and
\begin{equation}
\b = \frac {|w|}{r}.
\end{equation}

\bi The spine of the hypersurface is a circle of radius $ra$ centered at ${\mathrm C} =\pi \e_0$.
As $a$ runs through the positive reals, the corresponding family of hypersurfaces
foliates the ``punctured" $\chtwo$ which would result from the removal
of the complex projective line passing through ${\mathrm C}$ and $\pi \e_2$.
As before, each $M$ has three distinct smooth principal curvatures
except on the spine where two principal curvatures are zero and
one is $\a\ne 0$.  In this case, $\a$ cannot vanish, so that the principal curvature
functions are smooth everywhere.

As $a$ approaches $\infty$, $\a$ approaches $-\frac{2}{r\s} $ from below.
Thus $\a$ approaches the value for hypersurfaces based on the
horocycle which are explained in the next section.
\vspace{0.5cm}

As in the previous section, all extended rulings pass through $\pi \e_2$.  In fact, the complex projective line $\ell$ in which the curve $\pi\circ\bdelta$ lies passes through $\pi \e_0$ and $\pi \e_1$.  For $\bp = r (\cosh\ a\ \e_0 + \sinh\ a\ \e_1) \in H^5_1(r)$, let $\bq = r(\sinh\ a\ \e_0 + \cosh\ a\ )\e_1$.  The tangent space to $H^5_1(r)$ at $\bp$ is spanned by the orthonormal set
$\{{\mathrm i}\bp, \bq, {\mathrm i}\bq, \e_2, {\mathrm i}\e_2\}$.  Thus, the tangent space to $\chtwo$ at $\pi \bp$ is spanned by $\{ \pi_*\bq,\pi_*{\mathrm i}\bq, \pi_*\e_2, \pi_*{\mathrm i}\e_2\}$.  This splits into orthogonal subspaces  spanned by $\{ \pi_*\bq,\pi_*{\mathrm i}\bq\}$ and $\{ \pi_*\e_2, \pi_*{\mathrm i}\e_2\}$, respectively -- one tangent to the $\cpone$ in which $\pi\circ\bdelta$ lies, the other tangent to the $\cpone$ containing the ruling $\ell_\ppt ^\perp$ through $\ppt=\pi \bp$.

\bi \bi
\subsubsection* {Hypersurfaces generated by horocycles}

Again, we use the general formulas with $\teps = 1$ and $\epsilon = -1$.  Let
$$
 Y = \begin{bmatrix}
 {\mathrm i} & -{\mathrm i} \\
              {\mathrm i} &-{\mathrm i} & \\
\end{bmatrix}
$$
so that $Y^2 = 0$ and
$$\exp tY = \begin{bmatrix}
 1+t{\mathrm i} &\quad -t{\mathrm i} \\
             t {\mathrm i} &\quad 1-t{\mathrm i} & \\
\end{bmatrix}.
$$
 Let $a$ be an arbitrary real number and take
$$
\bp = r (\cosh\ a\ \e_0 + \sinh\ a\ \e_1) \in H^5_1(r).
$$
Then
\begin{align}
{\mathrm i}\bp &= r {\mathrm i} (\cosh\ a\ \e_0 + \sinh\ a\ \e_1)\\
Y\bp &= r{\mathrm i} (\cosh\ a\  - \sinh\ a) (\e_0 + \e_1).
\end{align}

Let $\bdelta(s) = (\exp (s\l Y))\bp$ so that $\bdelta_s  = \l (\exp (s\l Y))Y \bp$
and $\bdelta_{ss} = \l^2(\exp (s\l Y)) Y^2\bp =0.$  
Since $|\e_0+\e_1| =0$, we get $\<\bdelta_s, \bdelta_s\> = 0$.  Furthermore, a straightforward calculation gives 
$b = - \l$ and 
$\<\bdelta_s, \bdelta_s\> + b^2 r^2 = \l^2 r^2.$
 Choose $\l = 1/r$.  Then we have
\begin{equation}
\a = - \frac{2}{r \s}
\end{equation}
and
\begin{equation}
\b = - \frac {|w|}{r}.
\end{equation}

\bi
The spine $\g$ of $M$ is a horocycle since it has constant geodesic curvature $2/r$.  Since $\a$ is nonzero, the principal curvature functions are smooth.  On the spine, two principal curvatures are zero and the third is equal to $\a$.  At all other points, the principal curvatures are distinct.

As $a$ varies, we get a family of horocycles, just as a point determines a family of circles and a line (geodesic) determines a family of equidistant curves.  The hypersurfaces whose spines belong to this family of horocycles will also foliate $\chtwo$.

Finally we note that all the extended rulings of the hypersurface $M$ pass through $\pi \e_2$.

\subsubsection* { Hypersurfaces generated by circles outside of $\chtwo$}

In this case, we have $\teps = -1$ and $\epsilon= -1$.  Our hypersurface will be built over circles 
on the $\cptwo$ which does not intersect the $\chtwo$ in which the hypersurface will lie.  
Although the computation formally resembles that for Case 1, the columns of the $2 \times 2$ matrices correspond to 
$\e_1$ and $\e_2$ rather than to
$\e_0$ and $\e_1$.

Let
$$
 Y = \begin{bmatrix}
  0 & {\mathrm i} \\
              {\mathrm i} & 0 \\
\end{bmatrix}
$$
so that
$$\exp tY = \begin{bmatrix}
 \cos t & {\mathrm i} \sin t \\
             {\mathrm i} \sin t & \cos t \\
\end{bmatrix}
$$
and $Y^2 = -I.$
Let $\l$ be a nonzero constant to be determined later.
Then $s \mapsto \exp (s\l Y)$ gives a 1-parameter group of transformations preserving $\< \ , \>$. We extend by having the transformations leave $\e_0$ fixed.
Let $a$ be a real number satisfying $-\frac{\pi}{4}<a<\frac{\pi}{4}$ and take
$$
\bp = r (\cos\ a\ \e_1 + \sin\ a\ \e_2) \in S^3(r).
$$
Then
\begin{align}
{\mathrm i}\bp &= r{\mathrm i} (\cos\ a\ \e_1 + \sin\ a\ \e_2)\\
Y\bp &= r{\mathrm i} (\sin\ a\ \e_1 + \cos\ a\ \e_2).
\end{align}

Let $\bdelta(s) = \exp (s\l Y) \bp$ so that $\bdelta_s  = \l (\exp (s\l Y))Y \bp$
and $\bdelta_{ss} = -\l^2 \bdelta.$  By a similar calculation, we get $\<\bdelta_s, \bdelta_s\> = \l^2r^2$,
$b = \l \sin 2a$, and 
$\<\bdelta_s, \bdelta_s\> - b^2 r^2 = r^2\l^2\cos^2 2a.$
We choose $\l = \frac{1}{r}\sec 2a$.
Thus,
\begin{equation}
\a = -\frac{2}{r \s} \tan 2a
\end{equation}
and
\begin{equation}
\b = -\frac {|w|}{r}.
\end{equation}

Note that the curve $\gam$ is a curve of constant geodesic curvature $|\a|$ so that varying $a$ produces a family of concentric circles centered at $C=\pi (\e_1 +\e_2)$ on the 
$\cpone$.  
When $a=0$ this circle is a geodesic of the $\cpone$.  As $a$ varies, $|\a|$ increases and tends to $\infty$ as $a$ approaches $\pm \frac{\pi}{4}$ and the radius approaches zero. 

As we have seen, the continuous principal curvature functions are $0$,  $\frac{1}{2} (\a + \sqrt{\a^2 + 4 \b^2})$ and
 $\frac{1}{2} (\a - \sqrt{\a^2 + 4 \b^2})$. 
In our case, 
$$
\a^2 + 4 \b^2 = \frac{4}{r^2} (|w|^2 + (|w|^2 - 1)\tan^2 2a)
$$
which is always positive on the hypersurface since $|w|^2 > 1$.  Hence the principal curvatures are smooth functions and $M$ is a ruled hypersurface with three distinct principal curvatures.  In the special case where $a=0$, $M$ is minimal with principal curvatures $0$ and $\pm \b$. 

By an argument similar to that used in the preceding cases, we see that all extended rulings pass through $\pi \e_0 \in \chtwo$.  
This point cannot, however, be on the hypersurface as one can see from \eqref{basic-z}.
The hypersurfaces constructed in this subsection foliate the punctured $\chtwo$ obtained by removing the $\cpone$ that passes through $C$ and $\pi \e_0$. 

\bi

 \subsection{Group Actions with codimension-one orbits in $\Mh$}\label{actions}

In this section, we show that the hypersurfaces constructed in \S \ref{cp2Examples} and \S \ref{ch2Examples} are of cohomogeneity one.  Specifically, if
$L$ is the complex projective line in which the curve $\gam$ lies, there is a one-dimensional group of transformations of $L$ that preserves $\gam$ (rotations, translations, or parallel displacements, depending on the curvature of $\gam$.)
On the other hand, 
for each point $\ppt=\pi \bp$ of $\gam$, there is a group of rotations about $\ppt$ that preserves the ruling through $\ppt$.  In fact, as we shall see, the same 1-parameter group of rotations preserves each ruling.

Lifting each of these groups of transformations to $\X'$, we get a 2-dimensional group $G'$ of isometries of $\X'$ that preserves $M'$ and whose orbits are cylinders or tori, codimension one in $M'$). 

Specifically, for the $\cptwo$ case, the first 1-dimensional group comes from the following orthogonal transformation of 
$\R^6$,
\begin{equation}
\begin{bmatrix}
\cos u & - \sin u & 0 & 0 &0 &0 \\
\sin u & \cos u & 0 &0 &0 &0\\
 0 &0 &\cos u &\sin u &0 &0                \\
0    & 0 &-\sin u &\cos u &0 &0     \\
 0 &0 &0 &0 &1  &0\\
  0 &0 &0 &0 &0 &1
\end{bmatrix},
\end{equation}
and the second from
\begin{equation}
\begin{bmatrix}
1 &0 &0 &0 &0 &0                \\
0 &1  &0 &0  &0 &0     \\
 0 &0 &1 &0 &0  &0\\
  0 &0 &0 &1 &0 &0 \\
0 &0 &0 &0 &\cos v &-\sin v \\
 0 &0 &0 &0 &\sin v &\cos v.
 \end{bmatrix}.
\end{equation}
where $u, v \in \R$.
These are matrix forms of transformations $\Omega_u$ and $\Theta_v$ with respect to the basis $\{\e_0, {\mathrm i}\e_0, \e_1, {\mathrm i}\e_1, \e_2, {\mathrm i}\e_2\}$.  More simply, 
\begin{equation}
\Omega_u \e_0 = e^{{\mathrm i}u}\e_0,\quad \Omega_u \e_1 = e^{-{\mathrm i}u}\e_1,\quad \Omega_u \e_2 = \e_2,
\end{equation}
and
\begin{equation}
\Theta_v \e_0 = \e_0,\quad \Theta_v \e_1 = \e_1,\quad \Theta_v \e_2 = e^{{\mathrm i}v}\e_2.
\end{equation}

Note that $\Omega_u$ and $\Theta_v$ commute.  Also $\Omega_u\circ \Omega_{\tilde u}  = \Omega_{u+{\tilde u}}$ and $\Theta_v\circ \Theta_{\tilde v}  = \Theta_{v+{\tilde v}}$.  We let $G$ be the group of isometries of $\X$ generated by the projections of all the $\Omega_u$ and $\Theta_v$.  

For $\chtwo$, the situation is analogous.  For the three kinds of  curves in $\chone$, we have
\begin{enumerate}
\item
equidistants: 
\begin{equation}
\Omega_u \e_0 = \cosh u\ \e_0 + \sinh u\ \e_1,\quad \Omega_u \e_1 = \sinh u\ \e_0 + \cosh u\ \e_1,\quad \Omega_u \e_2 = \e_2;
\end{equation}
\item
circles:
\begin{equation}
\Omega_u \e_0 = e^{{\mathrm i}u}\e_0,\quad \Omega_u \e_1 = e^{-{\mathrm i}u}\e_1,\quad \Omega_u \e_2 = \e_2;
\end{equation}
\item
horocycles:
\begin{equation}
\Omega_u \e_0 = (1+u{\mathrm i})\ \e_0 + u{\mathrm i}\ \e_1,\quad \Omega_u \e_1 = -u{\mathrm i}\ \e_0 + (1-u{\mathrm i})\ \e_1,\quad \Omega_u \e_2 = \e_2.
\end{equation}
\end{enumerate}
In all three cases, we still have
\begin{equation}
\Theta_v \e_0 = \e_0,\quad \Theta_v \e_1 = \e_1,\quad \Theta_v \e_2 = e^{{\mathrm i}v}\e_2.
\end{equation}

For $\chtwo$ when $\gam$ is a circle outside $\chtwo$, we have
\begin{equation}
\Omega_u \e_0 = \e_0,\quad \Omega_u \e_1 = \cos u\ \e_1 + {\mathrm i} \sin u\ \e_2,\quad \Omega_u \e_2 = {\mathrm i} \sin u\ \e_1 + \cos u\ \e_2
\end{equation}
and
\begin{equation}
\Theta_v \e_0 = e^{iv} \e_0,\quad  \Theta_v \e_1 = \e_1,\quad \Theta_v \e_2 = \e_2.
\end{equation}

All these transformations satisfy the same multiplication rules as are satisfied in the $\cptwo$ case.  Furthermore, 
in the context of the general calculation for ruled hypersurfaces over curves in $\cpone$ and $\chone$ in \S\ref{construction}, 
if we let $\bdelta(s) = \exp(s\l Y)\bp$, so that $\pi \circ\bdelta$ is the curve of constant geodesic curvature used in our constructions in \S \ref{cp2Examples} and \S \ref{ch2Examples}, then 
\begin{equation}
(\Omega_u \circ \Theta_v) f(\theta, s, w)= f(\theta, s+\frac{u}{\lambda}, e^{iv} w).
\end{equation}
Therefore, $G'$ preserves $M'$.
Let $G$ be the corresponding group of isometries of $\cptwo$ (resp. $\chtwo$) and also of $M$.  Then for almost all $\qpt \in M$, the isotropy group $G_\qpt$ is finite so the orbits are two-dimensional. Thus, $M$ will be a hypersurface of cohomogeneity 
one.

\subsubsection* {The transversality condition}

We now show that the examples constructed in \S \ref{cp2Examples} and \S \ref{ch2Examples} satisfy the transversality condition.  Specifically, for any point $\ppt$ in $M$ there is an orbit through $\ppt$ that is not tangent to the ruling through $\ppt$.

It is enough to look at the $\Omega_u$ as a 1-parameter subgroup through $u=0.$   Consider an arbitrary point $\ppt=\pi \bz$ of $M$ where $\bz$ satisfies \eqref{basic-z}, i.e.
$$
\bz = f(\theta, s, w) = e^{i \theta} \s (\bdelta(s) + w r \bv).
$$
Then, as seen above,

\begin{equation}
\Omega_u \bz =f(\theta, s+\frac{u}{\lambda}, w)= e^{i \theta} \s (\bdelta(s+\frac{u}{\lambda}) + w r \bv).
\end{equation}
Differentiating this expression and setting $u=0$, we find the initial tangent vector to the orbit $\Omega_u\cdot\bz$ to be
$\frac{1}{\l} \ e^{i \theta} \s  \bdelta_s.$

Now using \eqref{secondeformula}, we get
\begin{equation}
e^{i \theta} \s \bdelta_s = \be + e^{i \theta} \s \teps\epsilon b {\mathrm i} \bdelta.
\end{equation}
It is easy to check that $e^{i \theta}{\mathrm i} \bdelta$ is orthogonal to $\be$.  Recall that $\{{\mathrm i} \bz, \be, {\mathrm i} \bR, \bR\}$ is an orthogonal spanning set for the tangent space to $M'$ at $\bz$.  Applying $\pi_*$ to these vectors yields $0$ and positive multiples of $W$, 
$E_1$ and $E_2$ respectively.  Thus $\pi_*(e^{i \theta} \s \teps\epsilon b {\mathrm i} \bdelta)$ must lie in the span of $E_1$ and $E_2$. As a result,
the initial tangent vector to the orbit of the projection of $\Omega_u$ includes a nonzero multiple of $\pi_* (\be + e^{i \theta} \s \teps\epsilon b {\mathrm i} \bdelta)$ and therefore cannot be orthogonal to $W$.  We must conclude that $\Omega_u$ acts tranversely on $M$ at $\ppt$.  We state this as a proposition.

\begin{prop} \label{transverse-prop}
Let $M$ be one of the ruled cohomogeneity-one hypersurfaces constructed in \ref{cp2Examples} and \ref{ch2Examples} and let $G$ be the group of isometries of $\X$ generated by the $\Omega_u \circ \Theta_v$ as described above.  Then $G$ acts transversely on $M$, i.e. for each $\ppt \in M$ the tangent space at $\ppt$ to the orbit 
$G\cdot \ppt$ does not lie in the tangent space at $\ppt$ to the ruling through that point. 
\end{prop}

 \section{\revisemorestill{Proof of Theorem 1}}\label{Moving-Frames}
\newcommand\ve{{\mathsf e}}
\newcommand\vv{{\mathsf v}}
\newcommand\vW{W} \newcommand\bE{{\mathbf E}}
\newcommand\bQ{{\mathbf Q}}
\newcommand\Zh{{\widehat Z}} \newcommand\Rhat{{\widehat R}} \newcommand{\fm}{\mathit{m}}
\newcommand\ellhat{{\widehat \ell}} \renewcommand\:{\mspace{2mu}}
\newcommand{\duals}{\varpi}

\newcommand\bigG{\EuScript G}
\newcommand\calG{\mathcal G}
\newcommand\Mho{\mathring{\Mh}}
\renewcommand\qpt{{\mathrm q}}
\newcommand\primept{\ppt'}
\newcommand\zpt{{\mathrm z}}  \newcommand\zprimept{\zpt'}
\newcommand\Z{{\mathscr Z}} \renewcommand\aa{\hat\alpha}  \newcommand\bb{\hat\beta} \newcommand\nablat{{\widetilde \nabla}} \newcommand\sg{\phi} \newcommand\sigmah{{\underline \sg}} \newcommand{\spanc}{\operatorname{span}_\C}
\newcommand{\spanr}{\operatorname{span}_\R}

\subsection{Moving Frames and Exterior Differential Systems}\label{movingdes}
We will prove Theorem \ref{main-theorem} using the techniques of moving frames
and exterior differential systems\footnote{Briefly, an {\bf exterior differential system} on a manifold
is an ideal $\I$ in the algebra of smooth differential forms on the manifold, such that $\I$ also closed under exterior differentiation.
Submanifolds to which all the differential forms in $\I$ pull back to be zero are {\bf integral submanifolds} of the ideal.
When an independence condition is present, represented locally by a decomposable $n$-form $\Omega$ which is well-defined up to 
adding $n$-forms in $\I$, attention is usually confined to {\bf admissible} integral submanifolds, i.e., those $n$-dimensional submanifolds to which $\Omega$ restricts to be a volume form.} which provide us with a systematic way of determining the satisfiability of the multiple conditions we are imposing on our hypersurfaces.  Background material in this subject may be found in the
textbook \cite{CfB2}.

An orthonormal frame $(\ve_1, \ve_2, \ve_3, \ve_4)$
at a point in $\Ma$ is defined to be {\em unitary} if $\JJ \ve_1 = \ve_2$ and $\JJ \ve_3 = \ve_4$.
We let $\F$ be the bundle of unitary frames on $\Ma$.  As a sub-bundle of the orthonormal frame bundle, $\F$ carries canonical 1-forms $\w^a$ and connection forms $\w^a_b$ for $1\le a,b \le 4$.  
These 1-forms are characterized by the properties that, if $\s$ is any local section of $\F$, and $(\ve_1, \ve_2, \ve_3, \ve_4)$ is the 
corresponding local unitary frame field on $\Ma$, then the $\s^*\w^a$ comprise the dual coframe field
and the Levi-Civita connection $\nablat$ on $\Ma$ satisfies
$$\langle \ve_a, \nablat_\bv \ve_b \rangle = \vv \intprod \s^* \w^a_b$$
where $\vv$ is any tangent vector on $\Ma$.  The connection forms satisfy $\w^a_b = -\w^b_a$ but also, because
the complex structure $\JJ$ is parallel with respect to $\nablat$, we have
$$\w^3_1 = \w^4_2, \quad \w^3_2 = -\w^4_1.$$
The 1-forms $\w^1, \ldots, \w^4, \w^2_1, \w^4_1, \w^4_2, \w^4_3$ are pointwise linearly independent  and give
a globally defined coframe on $\F$.  Moreover, they satisfy structure equations 
\begin{equation}\label{dcanon}
d\w^a = \w^b \wedge \w^a_b
\end{equation}
and
$
d\w^a_b + \w^a_k \wedge \w^k_b = \Phi^a_b,
$
where 
\begin{align*}
\Phi^1_2 &= c(4\w^1 \wedge \w^2 +2\w^3 \wedge \w^4)\\
\Phi^3_4 &= c(2\w^1 \wedge \w^2 +4\w^3 \wedge \w^4)\\
\Phi^2_4 &= c(\w^1 \wedge \w^3 +\w^2 \wedge \w^4) =\Phi^1_3\\
\Phi^1_4 &= c(\w^1 \wedge \w^4 - \w^2 \wedge \w^3) =-\Phi^2_3.
\end{align*}

Let $\Mh$ be a hypersurface in $\Ma$.  Near any point of $\Mh$ there is an open set $\setU$ and a local section $\sg:\setU \to \F$ 
defining a unitary frame field on $\setU$ such that  $\ve_3 = \vW$, the structure vector.  Since $\ve_4$ would then be normal to the hypersurface,
$\sg^*\w^4=0$ and by differentiation the connection forms $\w^4_i$ satisfy
$\sg^*\w^4_i = A_{ij} \w^j$, where $1 \le i,j \le 3$ and $A_{ij}$ are the components of the shape operator of $\Mh$ (see \cite{crbook}, p.445).

Suppose now that $\Mh$ is ruled.  As noted in \S\ref{background}, the functions  $\alpha = \langle AW, W\rangle$ and $\beta = |AW - \alpha W|$ are globally-defined functions on $\Mh$, and near points where $\beta >0$ there is a unique unitary frame field such that 
$\ve_3=W$ and 
\begin{equation}\label{Aadapt}
A \vW = A \ve_3 = \alpha \ve_3 + \beta \ve_1.
\end{equation}
Moreover, since it is impossible for $\beta$ to vanish on an open subset of $\Mh$, then $\beta$ is positive on an open dense subset
$\Mho \subset \Mh$.  On $\Mho$, functions $\alpha$ and $\beta$ are smooth, as is the unitary frame field satisfying \eqref{Aadapt}.
It follows that
\begin{equation}\label{Aofees}
A \ve_1=\beta \ve_3, \quad A \ve_2 = 0,
\end{equation}
on $\Mho$, and the pullbacks of the connection forms satisfy
\begin{equation}\label{Apullback}
\begin{aligned}
\sg^*\w^4_1 &= \beta\: \sg^* \w^3 \\
\sg^*\w^4_2 &= 0\\
\sg^*\w^4_3 &= \beta\: \sg^* \w^1 + \alpha\: \sg^* \w^3.
\end{aligned}
\end{equation}

Our approach to this classification (and proving Theorem \ref{main-theorem}) will be based on associating $\sg$ with 
an integral submanifold of an exterior differential system (EDS). This submanifold is obtained by combining $\sg$ with the 
the components of the shape operator.  
More precisely, we adjoin to $\sg$ the values of the
functions $\alpha, \beta$, obtaining a mapping $\sigmah : \Mho \to \F \times \R^2$ defined
by 
\[
\sigmah: \mpt \mapsto (\mpt, \ve_1\restr_{\mpt}, \ldots, \ve_4\restr_{\mpt}, \alpha(\mpt), \beta(\mpt)).
\]
We will take $(\aa,\bb)$ as coordinates on second factor in $\F \times \R^2$ (so that $\sigmah^* \aa =\aa\circ \sigmah = \alpha$ and $\sigmah^* \bb =\bb\circ \sigmah= \beta$) and let $\W \subset \R^2$ be the half-plane where $\bb>0$.
Then the image of $\sigmah$ lies in $\F \times \W$, and it follows from $\sg^*\w^4=0$ and \eqref{Apullback} that this image is an integral submanifold of the Pfaffian EDS
\begin{equation}\label{defofI}
\I = \{ \w^4,  \w^4_1  - \bb\:\w^3, \w^4_2, \w^4_3 - \bb\:\w^1 - \aa\:\w^3 \}_{\text{diff}}
\end{equation}
on $\F \times \W$.
On the other hand, given a 3-dimensional integral submanifold $S$
satisfying the independence condition, the image of $S$ under the 
basepoint projection of $\F$ is a smooth hypersurface $\Mho$ in 
$\Ma$, and it follows from the vanishing of the 1-forms of $\I$
that the shape operator of $\Mho$ satisfies \eqref{Aadapt} and is
thus a ruled hypersurface.  More formally, we have:
\begin{prop}\label{prop-eds} Given a ruled hypersurface $\Mh$, let $\Mho \subset \Mh$ be the open dense subset where $\beta >0$, and let $
(\ve_1, \ldots \ve_4)$ be the unique unitary frame on $\Mho$ satisfying \eqref{Aadapt} and \eqref{Aofees}.  Then $\sigmah(\Mho)$ is an integral 
submanifold $S$ of $\I$ satisfying the independence condition $\w^1 \wedge \w^2 \wedge \w^3 \ne 0$.  
Conversely, given any such submanifold $S$ there is a ruled hypersurface $\Mho$ equipped with a unitary frame field satisfying \eqref{Aadapt}, \eqref{Aofees}, such that $\sigmah(\Mho) = S$.
\end{prop}

\subsection{Proof of classification results}\label{classproof}
\begin{lemma}\label{lemma-pis} 
Let $\Mh$ and $\Mho$ be as in Proposition \ref{prop-eds}
such that $\Mh$ has cohomogeneity one.
Let $\sigmah: \Mho \to \F \times \W$ be the integral submanifold given by Proposition \ref{prop-eds}, and on $\F \times \W$ define 1-forms
\begin{equation}\notag
\label{pidefs}
\begin{aligned}
\pi_1 &:=d\bb - (\bb^2 + c) \w^2,\\
\pi_2 &:= \bb \w^2_1 - (\bb^2-c) \w^1, \\
\pi_3 &:= d\aa - \aa\bb\: \w^2,
\end{aligned}
\end{equation}
where $c=\epsilon/r^2$.  Then 
$\sigmah^* \pi_1 = 0$, 
$\sigmah^* (\w^3 \wedge \pi_2) = 0$
and $\sigmah^* (\w^2 \wedge \pi_2+ \w^3 \wedge \pi_3) = 0.$
\end{lemma}

\begin{proof}
By Proposition \ref{prop-eds}, the image of $\sigmah$ is an admissible integral submanifold of the Pfaffian system $\I$ defined by \eqref {defofI}.  For the sake of convenience, we label the 1-forms of that system as
\begin{equation}\notag
\theta_0 := \w^4, \quad \theta_1 := \w^4_1 - \bb \w^3,\quad \theta_2 := \w^4_2,\quad 
\theta_3 :=\w^4_3 - \bb\w^1 - \aa \w^3.
\end{equation}
To determine further conditions satisfied by admissible integral submanifolds, we will need a complete set of algebraic generators
for $\I$.  These are given by $\theta_0, \ldots,\theta_3$ and their exterior derivatives, which we now compute.

First note that, from \eqref{dcanon},
\begin{equation}\notag
d\theta_0 = \w^1 \wedge \theta_1 +\w^2 \wedge \theta_2 +\w^3 \wedge \theta_3,
\end{equation}
so that $d\theta_0$ is already in the algebraic ideal generated by the 1-forms (and hence we have no need to include 
$d\theta_0 $ among the algebraic generators of $\I$); we express this more succinctly by writing
\begin{equation}\notag
d\theta_0 \equiv 0\quad \text{mod}\quad \theta_0, \ldots, \theta_3.
\end{equation}
We compute the other generator 2-forms as 
\begin{equation}\label{twoforms}
\begin{alignedat}{2}
d\theta_1 &\equiv && \w^3\& \pi_1,\\
d\theta_2 &\equiv && \w^3\&\pi_2,\\
d\theta_3 &\equiv \w^1 \& \pi_1 +\w^2\&\pi_2 \,+\,&&\w^3\&\pi_3,
\end{alignedat}
\end{equation}
modulo $\theta_0, \ldots, \theta_3$.

Since $S = \sigmah(\Mho)$ is an admissible integral submanifold of $\I$, each of the 2-forms on the right in \eqref{twoforms} must vanish along\footnote{In what follows, when we speaking of a differential form vanishing or an equation involving differential forms
holding `along $S$', we will mean that the forms satisfy that condition when pulled back to $S$ under the inclusion map.}
 $S$. By applying the Cartan Lemma (see Lemma A.1.9 in \cite{CfB2}) we see that there are functions $p, q, s$ on $S$ such that 
\begin{equation}\label{pivals}
\pi_1 = p\:\w^3, \quad \pi_2 = q\:\w^3, \quad \pi_3 = p\:\w^1 + q\:\w^2 + s\:\w^3
\end{equation}
along $S$.  By substituting in the definitions of $\pi_1,\pi_2,\pi_3$ it follows that 
\begin{subequations} \label {abval}
\begin{alignat}{3}
d\aa &= p\:\w^1 &+ (\aa\bb+q)&\w^2 &+s\: &\w^3,\label{daval} \\
d\bb &= &(\bb^2+c)&\w^2 &+p\:&\w^3, \label{dbval} \\
\bb\:\w^2_1 &= (\bb^2-c)\w^1 &&&+q\:&\w^3
\end{alignat}
\end{subequations}
along $S$.  

By the cohomogeneity assumption, there is a dense subset of $\Mh$ which is a union of two-dimensional orbits of a
subgroup $G$ of the ambient isometry group.   Since the shape operator and structure vector are preserved by such isometries, $\alpha$ and $\beta$ are constant along these orbits, and hence the differentials of $\alpha$ and $\beta$ are linearly dependent at each point of $\Mh$.  
Correspondingly, $d\aa$ and $d\bb$ must be linearly dependent along $S$.  Imposing this condition on the right-hand sides in \eqref{daval} and \eqref{dbval} shows that $p=0$ at each point of $S$.  Hence $\pi_1=0$ along $S$, and the vanishing along $S$ of the remaining 2-form
generators of $\I$ in \eqref{twoforms} gives the conclusion.
\end{proof}

\begin{prop}\label{prop-case1} Let $\Mh$, $\Mho$, $\pi_2$, $\pi_3$ be as in Lemma \ref{lemma-pis}.  Let $G$ be the group of isometries preserving $\Mh$, with the union of the two-dimensional orbits of $G$ being dense in $M$.  Assume that the action of $G$ satisfies the transversality condition.  Suppose $\setU \subset \Mho$ is a connected open set on which $\beta^2 + c \ne 0$ (again, this is vacuous if $c>0$).  Then
\begin{enumerate}[label={\rm (\arabic*)},ref={(\arabic*)},itemsep=0pt]
\item\label{pivanish1} $\sigmah^* \pi_2 = 0$ and $\sigmah^* \pi_3 = 0$ on $\setU$;
\item $\alpha^2/(\beta^2+c)$ is constant on $\setU$;
\item there is point $\zpt \in \CP^2$ such that 
each ruling in $\setU$ when extended
passes through $\zpt$;
\item the point $\zpt$ is fixed by isometries in $G$.
\end{enumerate}
In the case where $c<0$ and $\beta^2+c >0$, $\zpt$ lies in $\chtwo$; in the case where $c<0$ and $\beta^2+c <0$, $\zpt$ lies in 
$\cptwo_1$.
\end{prop}

\begin{proof} Part (1): Continuing the argument of the previous proof, the condition that $d\aa$ and $d\bb$ are linearly dependent along $S$, and the fact that $p=0$ and \revise{$\bb^2+c \ne 0$} along $S$, imply using \eqref{abval} that $s=0$ along $S$.
Then taking the exterior derivative of $d\bb = (\bb^2+c)\w^2$, and using the value for $\w_2^1$ along $S$ given by \eqref{abval}, we obtain
$
0= \dfrac{q(\bb^2+c)}{\bb} \w^1 \wedge \w^3
$
and conclude that $q=0$ along $S$.  Conclusion \ref{pivanish1} now follows.

\medskip
\noindent
Part (2):
Pulling back the first two equations in \eqref{abval} to $\setU$, using $q=0$, and pairing with the frame vector $\ve_2$ gives
$\ve_2(\alpha)= \alpha \beta$,  $\ve_2 (\beta) = \beta^2+c$
on $\setU$.  It follows that $d(\alpha^2 / (\beta^2+c)) = 0$ on $\setU$, and hence $\alpha^2/(\beta^2+c)$ is constant
on the connected set $\setU$.

\medskip
\noindent
Part (3): 
We will obtain our results using elements of complex projective geometry. 
Because $\Mh$ is ruled, 
it is foliated by open subsets of complex projective lines.   We will show that
there exists a common point of $\cptwo$ through which all these lines pass.

We will compute using complex projective moving frames (see, e.g., Chapter 3 in \cite{CfB2}).  
Let $\bigG$ be $U(3)$ or $U(1,2)$ depending on whether $\Ma$ is $\cptwo$ or $\chtwo$ respectively.  
Let $\bE_0, \bE_1, \bE_2$ be the columns of an element of $\bigG$, regarded as $\C^3$-valued functions on $\bigG$.
The complex-valued Maurer-Cartan forms $\psi^i_j$ on $\bigG$ (where $0 \le i,j \le 2$) are defined by
\begin{equation}\label{dees}
d\bE_i= \bE_j \psi_j^i.
\end{equation}
Recalling that $\epsilon = \operatorname{sign}(c)$, these satisfy relations
$$\psi^j_i = \begin{cases}
-\epsilon\overline{\psi^i_j} & \text{if exactly one of $i,j$ is equal to zero,}\\
-\overline{\psi^i_j} & \text{otherwise,}
\end{cases}$$
which mean that the matrix-valued 1-form $(\psi^i_j)$ takes value in the Lie algebra of $\bigG$.

Define a submersion $\Pi:\bigG \mapsto \F$ as follows: if $g=(\bE_0, \bE_1, \bE_2)\in \bigG$, then $\Pi(g)$ is the unitary frame with basepoint 
$\pi(\bE_0)$ and components
\begin{equation}\label{eident}
\ve_1 = \pi_* \bE_1, \quad \ve_2 = \pi_* \ri \bE_1, \quad \ve_3 = \pi_* \bE_2,\quad \ve_4 = \pi_* \ri \bE_2.
\end{equation}
This mapping makes $\bigG$ into a principal $S^1$-bundle over $\F$, since the action $\bE_0 \mapsto e^{\ri\theta}\bE_0$ is simply transitive on the fibers.  We will trivially extend $\Pi$ to a mapping from $\bigG \times \W$ to $\F \times \W$ by applying the identity on the second factor.  For use below, we will need the following relationships (derived in \S3.1 of \cite{dalembert}) between the Maurer-Cartan forms of $\bigG$ and the pullbacks of the canonical and connection forms on $\F$:
\begin{equation}\label{pully}
\begin{aligned}
\Pi^*(\w^1 + \ri \w^2) &= r\psi^1_0, & \quad \Pi^*(\ri \w^2_1) &=\psi^1_1 - \psi^0_0,\\
\Pi^*(\w^3 + \ri \w^4) &= r\psi^2_0, & \Pi^*(\ri \w^4_3) &= \psi^2_2 - \psi^0_0, \\
\Pi^*(\w^3_1 + \ri\w^4_1) &= \psi^2_1, & 
\end{aligned}
\end{equation}
where $r$ is the constant such that $c=\epsilon/r^2$.
We will  regard the $\psi^a_b$ as 1-forms on $\bigG \times \W$, suppressing the pullback from $\bigG$.

From \eqref{pully} we see that the pullbacks under $\Pi$ of the factors in the independence condition 
$\Omega = \w^1\wedge \w^2 \wedge \w^3 \ne 0$ will be given by the wedge product of the real and imaginary parts of $r\psi_0^1$ and the real part of $r\psi_0^2$.  Thus, for the sake of convenience we define the real-valued 1-forms
\begin{equation}\label{etadefs}
 \eta^1 :=r\realpart\psi^1_0 = \Pi^* \w^1, \quad \eta^2 :=r\imagpart\psi^1_0=\Pi^*\w^2, \quad \eta^3 = r\realpart\psi^2_0
=\Pi^* \w^3\end{equation}
on $\bigG\times \W$.  Then since $\psi_0^0$ is nonzero on the fibers of $\Pi$, the 1-forms $\eta^1$, $\eta^2$, $\eta^3$ and 
$\psi_0^0$ restrict to give a coframe on the 4-dimensional submanifold $\Pi^{-1}(S) \subset \bigG \times \W$.

To show that the extensions of all the rulings in $\setU$ pass through a common point, we will 
define a certain $\C^3$-valued function on $\Pi^{-1}(S)$ giving a vector which lies in the deprojectivization $\pi^{-1}(\ell)$ of each extended ruling $\ell$ of $\setU$.  By differentiating this vector-valued function we will show that its span (over $\C$) is fixed as one moves across $S$. 
Thus, the image of this vector under projectivization $\pi$ is a single point $\zpt \in \cptwo$ through which each extended ruling passes.

To facilitate our calculation, we first note that the vanishing along $S$ of the generator 1-forms of $\I$ in \eqref{defofI}, as well as the vanishing of the 1-form $\pi_2$ given by Lemma \ref{lemma-pis}, imply using \eqref{pully} that along 
$\Pi^{-1}(S)$ the Maurer-Cartan forms satisfy the relations
\begin{equation}\label{psivals}
\imagpart \psi_0^2=0,\quad \psi_2^1=\ri \bb\eta^3, \quad \psi_1^1 = \psi_0^0 +
\ri \left(\frac{\bb^2-c}{\bb}\right)\eta^1, \quad\psi_2^2=\psi_0^0 + \ri(\bb\eta^1+\aa\eta^3).
\end{equation}
Using \eqref{dees} and the fact that $d\bb=(\bb^2+c)\w^2$ along $S$, we calculate that along $\Pi^{-1}(S)$, 
\begin{equation}\notag
d(\bE_1-\ri r\bb \bE_0)=\bE_0\psi_1^0 + \bE_1\psi_1^1 + \bE_2\psi_1^2 - \ri r(\bb^2+c)\bE_0 \eta^2 -\ri r\bb (\bE_0\psi_0^0 + \bE_1\psi_0^1 + \bE_2\psi_0^2). 
\end{equation}
Using \eqref{etadefs}, \eqref{psivals}, $\psi_1^0=-\epsilon\overline{\psi_0^1}$ and simplifying, we obtain
\begin{equation}\label{diffEE}
d(\bE_1-\ri r\bb \bE_0)
= (\bE_1 -\ri r \bb \bE_0) \bigg(\psi_0^0 - \frac{\ri c}{\bb} \eta^1+\bb\eta^2 \bigg)
\end{equation}
along $\Pi^{-1}(S)$.
Thus, the vector-valued function $\bE_1-\ri r\bb \bE_0$ on $\Pi^{-1}(S)$ is fixed up to a complex multiple.  If we let $\Zh$ be the complex 1-dimensional linear subspace of $\C^3$ spanned by $\bE_1-\ri r\bb \bE_0$, it follows that this subspace is constant on $\Pi^{-1}(S)$.

For a given point $\mpt \in \setU$, let $\ell_\mpt$ be the extended ruling through $\mpt$ and let $\ellhat_\mpt =\pi^{-1}(\ell_\mpt)\subset \C^3$.  Because $\ve_1$ and $\ve_2$ are tangent to $\ell_\mpt$ at $\mpt$, then due to \eqref{eident} and the fact that at each point along the fiber $\Pi^{-1}(\sigmah(\mpt))$ we have 
$\pi(\bE_0) = \mpt$, we conclude that $\ellhat_\mpt$ equals the complex span of the vectors $\bE_0$ and $\bE_1$ .  It follows that the fixed subspace $\Zh$ lies in $\ellhat_\mpt$ for all $\mpt \in \Mh$.  Hence, all the extended rulings $\ell_\mpt$ pass through the point $\zpt =\pi(\Zh) \in \cptwo$.
Because 
\begin{equation}\notag
\< \bE_1-\ri r \bb \bE_0, \bE_1 - \ri r \bb \bE_0\> = \frac{\bb^2+c}{c},
\end{equation}
we see that for $\Ma = \chtwo$ (i.e., $c<0$) this common point $\zpt$ is in $\chtwo$ when $\beta^2+c>0$ but lies 
in the  $\cptwo_1$ when $\beta^2+c <0$.  Thus, if we let $\teps=1$ when $\zpt \in \cptwo_1$ and $\teps=-1$ when $\zpt \in \chtwo$, we have
\begin{equation}\label{betasign}
\operatorname{sign}( \beta^2+c) = \epsilon \teps.
\end{equation}

\medskip
\noindent
Part (4): Because of the transversality condition, $G$ must fix the point $\zpt$ through which all the rulings in $\setU$ pass.
\end{proof}

\begin{proof}[Proof of Theorem \ref{main-theorem}]

\noindent Part (i):  We need only consider the case $c<0$.  
Since $\beta^2+c$ is a real-analytic function on $\Mh$, and not identically zero, it is nonzero on a dense open subset of $\Mh$.  Moreover, it cannot change sign on $\Mh$.  

Otherwise, since $\Mh$ is connected there is a real-analytic curve $\tau:I \to \Mh$ connecting points where $\beta^2+c$ has opposite signs.  By shrinking the domain of $\tau$ if necessary and reparametrizing, we can assume that $I=[-h,h]$ with $\beta^2+c<0$ at $\tau(t)$ for $t\in [-h,0)$, $\beta^2+c>0$ at $\tau(t)$ for $t \in (0,h]$, and that $\beta>0$ along the image of $\tau$.

By Proposition \ref{prop-case1}, there is a point $\zpt^+ \in \chtwo$ such that
for every $t \in (0,h]$ the ruling through $\tau(t)$ passes through $\zpt^+$, 
and a point $\zpt^- \in \cptwo_1$ such that 
for every $t\in [-h,0)$ the ruling through $\tau(t)$ passes through $\zpt^-$.  Moreover, from Proposition \ref{prop-case1} the isometry group $G$ that preserves $M$ fixes both $\zpt^+$ and $\zpt^-$.  

Because $\beta>0$, there is a real-analytic orthonormal frame $(W,X,Y)$ along the image of $\tau$ with respect to which the shape operator takes the form \eqref{ruled-shape}.  In particular,  $\{X,Y\}$ spans the ruling through each point $\tau(t)$.  The fact that the extended ruling through $\tau(t)$ passes through $\zpt^+$ for $t>0$ implies, by analyticity, that the same condition holds for all $t \in I$.  Similarly, the extended ruling through $\tau(t)$ passes through $\zpt^-$ for all $t\in I$.  In other words, this extended ruling is the unique projective line containing $\zpt^+$ and $\zpt^-$  and the image of $\tau$ is contained in this line.  Since $G$ fixes both $\zpt^+$ and $\zpt^-$, then it stabilizes the whole line.  But this contradicts the fact that $G$ acts transversely to the rulings at each endpoint of $\tau$.

\medskip
\noindent
Part (ii): 
Let $\setU \subset \Mho$ be the open set where $\beta^2+c\ne 0$ and $\beta>0$, which is dense in $\Mh$.  We claim that there is a single point  $\zpt \in \chtwo \cup \cptwo_1$ through which all the extended rulings in $\Mho$ pass.  For, suppose $\setU_1, \setU_2$ are two connected components of $\setU$. By applying Proposition \ref{prop-case1} to each of these, we obtain a point $\zpt_1$ through which all the extended rulings of $\setU_1$ pass, and a point $\zpt_2$ through which all extended rulings of $\setU_2$ pass.  (By Part (i), either both these points are
in $\chtwo$ or both are in $\cptwo_1$.)  If these are two distinct points, then applying the argument in Part (i) \revise{along a path} from a point in $\setU_1$ to a point in $\setU_2$ leads to a contradiction to the transversality of the action of $G$.

Let $\Zh \subset \C^3$ be the complex 1-dimensional subspace  $\pi^{-1}(\zpt)$,
let $\scV \subset \C^3$ be its orthogonal complement, and let 
$\fm = \pi(\scV)$, a complex projective line in $\cptwo$.
Although $\scV$ and $\Zh$ are fixed subspaces, each can be expressed as the span of certain combinations of the vector-valued functions $\bE_i$ which were defined above in the proof of Proposition \ref{prop-case1} and are evaluated along the fiber $\Pi^{-1}(\sigmah(\mpt))$ for a given point $\mpt\in \Mho$.  From that proof, we have 
$$\Zh = \spanc \{ \bE_1 - \ri r \beta \bE_0 \}.$$
It is then easy to check that its orthogonal complement is
$\scV = \spanc \{ \bE_0 - \ri \epsilon r \beta \bE_1, \bE_2 \}.$
As in the same proof, let $\ell_{\mpt}$ be the extended ruling through $\mpt$; then $\ellhat_{\mpt} = \pi^{-1}(\ell_{\mpt}) = \spanc \{ \bE_0, \bE_1\}.$
It follows that $\ell_{\mpt}$ intersects $\fm$ at $\qpt=\pi(\bE_0 - \ri \epsilon r \beta \bE_1)$.
Moreover, since the tangent space to $\ell_{\mpt}$ at $\qpt$ is spanned over $\R$ by $\{\pi_* (\bE_1 - \ri r \beta \bE_0), \pi_* (\ri \bE_1 + r\beta\bE_0)\}$,
then $\ell_{\mpt}$ is perpendicular to $\fm$ at $\qpt$.

\medskip
\noindent
Parts (iv) and (v):  
Since $c = \epsilon/r^2$, 
\begin{equation}\label{normy}
\< \bE_0 - \ri \epsilon r \beta \bE_1, \bE_0 - \ri \epsilon r \beta \bE_1\> = \epsilon + \epsilon^2 r^2 \beta^2 =  r^2(\beta^2+c).
\end{equation}
Suppose that $\Ma = \chtwo$, and assume that $\mpt$ is not one of the isolated zeros of $\beta^2+c$.
If $\beta^2+c<0$ at $\mpt$ then \eqref{normy} shows that $\qpt \in \chtwo$, while if $\beta^2+c>0$ at $\mpt$ then $\qpt \in \cptwo_1$.
In the latter case, since $\<\ ,\ \>$ restricts to $\scV$ to be positive definite, $\fm$ lies inside $\cptwo_1$.

\newcommand\handy{h}  \medskip
\noindent
Part (iii):
It remains to show that the curve $\gam$ traced out by the point $\qpt$ where the extended rulings intersect $\fm$ has constant geodesic curvature.  

As in \eqref{betasign} we choose $\teps=\pm1$ so that $\epsilon\teps$ is the sign of $\beta^2+c$.
Since $\qpt = \pi(\bE_0 - \ri \epsilon r \beta \bE_1)$, we will normalize the vector in parentheses, letting 
$$\bQ=\frac{1}{\handy}\,(\bE_0 - \ri \epsilon r \beta \bE_1), 
\qquad \text{where }\handy:=\sqrt{\epsilon \teps (\b^2+c)}.$$
One can check that $\bQ$ takes value in the 5-dimensional hypersurface
$$\Z' = \{ \bz \in \C^3 \mid \< \bz, \bz \> = \epsilon \teps r^2\}.$$
When $\teps=1$ then $\Z'$ is the same as $\X'$, and in any case this is the space in which
the curve $\bdelta$ used in our construction in \S\ref{construction} takes value.
We will determine the geodesic curvature of $\gam$ by computing the relevant component of the shape operator of the surface
$\Sigma = \pi^{-1}(\gam) \cap \Z'$.   

We begin by computing the derivative of $\bQ$ as a function on 
$\Pi^{-1}(S)$:
$$d\bQ = \ri \bQ (\imagpart \psi^0_0 + \beta \eta^1) + \teps r \handy \bE_2 \eta^3.$$
This shows that $\{ \frac{1}{r}\ri \bQ,  \bE_2 \}$ is an orthonormal basis for the tangent space 
to $\Sigma$ at $\bQ$, and $\duals^1 = \imagpart \psi^0_0 + \beta \eta^1$, $\duals^2 = \teps r \handy \eta^3$ are the dual 1-forms.  Since the unit vector $\bE_2$ is orthogonal to the fiber of $\pi$ through $\bQ$, then $\duals^2$ is the pullback of the arclength differential along $\gam$.  We further compute that
$$d\bE_2 = r \ri \bE_2 \duals^1 + \dfrac{\alpha}{\teps r \handy} \ri \bE_2 \duals^2 - \dfrac{\epsilon\teps}{r^2} \bQ \duals^2.$$
Then $\nablat \bE_2$ is computed by orthogonally projecting the vector part of this
vector-valued 1-form into $T_{\bQ} \X'$.  By pairing the 1-form part with $\bE_2$ we get
$$|\nablat_{\bE_2} \bE_2 | = \dfrac{|\alpha|}{r\handy} = 
\dfrac{|\alpha|}{r \sqrt{ |\beta^2+c|}},$$
which is constant due to Proposition \ref{prop-case1}. 
\end{proof}

\section{The constant norm condition}\label{normA} 
\beginrevise
\bi\me\par

As detailed in the introduction, Dom\'inguez-V\'azquez and
P\'erez-Barral classified ruled hypersurfaces in $\chn$ whose shape operator has constant norm. 
Besides the Lohnherr hypersurface (whose principal curvatures are all constant), they obtained a specific member of the range of examples described at the beginning of Section  \ref{ch2Examples}.  Here, for the sake of completeness, we reaffirm their result for the set of ruled hypersurfaces of cohomogeneity one.

\begin{corstar} 
\revisemorestill{
Let $M$ be one of the hypersurfaces constructed in Section \ref{coho-1-Construction}.  
Then the shape operator of $M$ has constant norm if and only if  $M$ lies in $\chn$ and is built, 
as described in Section \ref{coho-1-Construction}, on an equidistant curve in $\CH^1$ of 
geodesic curvature $\sqrt{-2c}$.}
\end{corstar}

Note that, to compare with the value of the geodesic curvature stated in \cite{dominguez} for the curve that produces the unique inhomogeneous ruled hypersurface satisfying the constant norm condition, one must keep in mind that
in \cite{dominguez} the holomorphic sectional curvature is $c$ while for us it is $4c$.
Moreover, when we distinguish among the possible curves of constant geodesic curvature in the real hyperbolic plane, the special curve we have been describing is an equidistant curve, not a circle.   In \cite {dominguez}, when the authors call this curve a ``circle", they are using the more general sense of the term as employed by Adachi, Maeda and others in their many papers about circles in complex space forms (see, for example, \cite{maeda}).

\begin{proof}
It is straightforward to show that of all the examples constructed in Section \ref{coho-1-Construction}, this is the only one whose shape operator has constant norm.  The details are as follows:

The shape operator has constant norm if and only if $\a^2 + 2\b^2$ is constant.  In each case, we have an explicit expression for $\a$ and $\b$ (see sections \ref{cp2Examples} and \ref{ch2Examples}).  The key non-constant quantities are $\s$ and $|w|$ which are related by
$$
\teps\s^2(1+\epsilon|w|^2) =1.
$$
We examine each of the cohomogeneity-one hypersurfaces \revisemorestill{in turn}.
\begin{enumerate}
\item
Hypersurfaces in $\cptwo$ over circles:  $\a = -\frac{2}{r \s} \cot 2a\ ; \b = -\frac {|w|}{r}\ ; \teps=\epsilon=1.$
\begin{align*}
 \a^2+2\b^2=\frac{4}{r^2 \s^2} \cot^2 2a + 2\ \frac {|w|^2}{r^2} &= \frac{2}{r^2}\big(2\cot^2 2a\ (1+|w|^2) + |w|^2\big) \\ 
&= \frac{2}{r^2}\big(2\cot^2 2a +  (2\cot^2 2a +1) |w|^2 \big).
\end{align*}
\item
Hypersurfaces in $\chtwo$ over equidistant curves:  $\a = -\frac{2}{r \s} \tanh 2a\ ; \b = \frac {|w|}{r}\ ; \teps=1\ ;\epsilon= -1.$
\begin{align*}
 \a^2+2\b^2=\frac{4}{r^2 \s^2} \tanh^2 2a + 2\ \frac {|w|^2}{r^2} &= \frac{2}{r^2}\big(2\tanh^2 2a\ (1-|w|^2) + |w|^2\big) \\ 
&= \frac{2}{r^2}\big(2\tanh^2 2a +  (-2\tanh^2 2a +1) |w|^2 \big).
\end{align*}
\item Hypersurfaces in $\chtwo$ over circles:  $\a = -\frac{2}{r \s} \coth 2a\ ; \b = \frac {|w|}{r}\ ; \teps=1\ ;\epsilon= -1.$
\begin{align*}
 \a^2+2\b^2=\frac{4}{r^2 \s^2} \coth^2 2a + 2\ \frac {|w|^2}{r^2} &= \frac{2}{r^2}\big(2\coth^2 2a\ (1-|w|^2) + |w|^2\big) \\ 
&= \frac{2}{r^2}\big(2\coth^2 2a +  (-2\coth^2 2a +1) |w|^2 \big).
\end{align*}
\item
Hypersurfaces in $\chtwo$ over horocycles:  $\a = -\frac{2}{r \s}\ ; \b = -\frac {|w|}{r}\ ; \teps=1\ ;\epsilon= -1.$
\begin{align*}
 \a^2+2\b^2=\frac{4}{r^2 \s^2}  + 2\ \frac {|w|^2}{r^2} &= \frac{2}{r^2}\big( 2(1-|w|^2) + |w|^2\big) \\ 
&= \frac{2}{r^2}\big(2  - |w|^2 \big).
\end{align*}
\item
Hypersurfaces in $\chtwo$ over circles outside $\chtwo$:  $\a = -\frac{2}{r \s} \tan 2a\ ; \b = -\frac {|w|}{r}\ ; \teps= -1\ ;\epsilon= -1.$
\begin{align*}
 \a^2+2\b^2=\frac{4}{r^2 \s^2} \tan^2 2a + 2\ \frac {|w|^2}{r^2} &= \frac{2}{r^2}\big(2\tan^2 2a\ (|w|^2-1) + |w|^2\big) \\ 
&= \frac{2}{r^2}\big(-2\tan^2 2a +  (2\tan^2 2a +1) |w|^2 \big).
\end{align*}
\end{enumerate}

Clearly, the norm of the shape operator cannot be constant unless the coefficient of $|w|^2$ in the expression for $\a^2+2\b^2$ is zero.  This can happen only in the second of the five cases and requires that $ 2 \tanh^2 2a =1$, i.e.,  $\sinh 2a = 1$.

The geodesic curvature $\kappa_g$ of a generating curve $\gam$ in $\chone$ is obtained by setting 
$w=0$ in the expression for $|\a|$ (or equivalently, setting $\s = 1$).   In Case (2), $\kappa_g^2 = 2/r^2$, while the $\chone$ in which $\gam$ lies has sectional curvature $4c = -4/r^2$, so that $\kappa_g = \sqrt{-2c}$.
\end{proof}

The argument given above works equally well in all dimensions even though we have concentrated in our construction on the $n=2$ case.  When a ruled hypersurface in $\cpn$ or $\chn$ is built over a plane curve in the same way, the shape operator acts trivially in all the extra dimensions so only $\a$ and $\b$ are relevant.

\endrevise

 \section{Concluding Remarks}

We began this paper seeking the ``simplest" ruled hypersurfaces in the nonflat complex space forms.
Our original idea was to choose a ``plane curve" and erect a perpendicular ruling over each point of the curve.  
We worked out the characteristic parameters $\a$ and $\b$ for the resulting hypersurface.
Within this class of hypersurfaces, we chose for further consideration those for which the plane curves with constant curvature were used.  We noted that these hypersurfaces exhibit a high degree of symmetry -- each is the union of two-dimensional orbits of an action of a group $G$ of isometries of the ambient space.

The question then arose -- are these the only such ``cohomogeneity-one" possibilities.  The answer was ``almost".  \revisemorestill{An additional example} emerged in the $\chn$ case, which we worked out in detail for $n=2$.  It is analogous to those already constructed, but uses a plane curve of constant curvature in the 
part of $\cpn$ that is exterior to $\chn$, where the K\"ahler metric is indefinite.

\revise{
This enlarged set of hypersurfaces turns out to be precisely those analytic ruled hypersurfaces of cohomogeneity one for which the orbits and the rulings meet transversely and $\b^2+c$ is not identically zero.}  Whether the transversality condition and the analyticity condition are necessary to characterize this class of hypersurfaces remains a question for further study, 
\revisemorestill{as does the classification of cohomogeneity-one ruled hypersurfaces with $\b^2 +c$ identically zero.}  

 \end{document}